\documentclass[11pt,leqeq]{article}
\usepackage{amssymb,amsthm}
\usepackage{amsmath}
\usepackage{amscd}
\usepackage{xy}
\xyoption{all}
\usepackage[dvips]{graphicx}
\newtheorem{thm}{Theorem}
\newtheorem{lem}[thm]{Lemma}
\newtheorem{cor}[thm]{Corollary}
\newtheorem{prop}[thm]{Proposition} 
\newtheorem{rem}[thm]{Remark}

\newtheorem{defn}[thm]{Definition}
\newtheorem{exmp}[thm]{Example}
\newtheorem{exc}[thm]{Exercise}

\date{}
\begin{document}
\setlength{\baselineskip}{16pt}
\title{On the Categorification of the M$\ddot{\mbox{o}}$bius Function}
\author{Rafael D\'\i az }
\maketitle

\begin{abstract}
In these notes we study several categorical generalizations of the M$\ddot{\mbox{o}}$bius function and discuss the relations between the various
approaches. We emphasize the topological and geometric meaning of these constructions.
\end{abstract}

\section{Introduction}

The M$\ddot{\mbox{o}}$bius function  $\mu: \mathbb{N}_+ \longrightarrow \mathbb{Z}$ is the map from the positive natural numbers to the integers such that $\mu(1)=1$,  $\mu(n)=0$ if $n$ has repeated prime factors,
and $\mu(n)=(-1)^k$ if $n$ is the product of $k$ distinct prime numbers.
It was introduced by  M$\ddot{\mbox{o}}$bius in 1832 and it has been since an important tool in number theory, complex analysis, and combinatorics. Our goal in these notes is to gently introduced several generalizations of the M$\ddot{\mbox{o}}$bius function
originating from the viewpoint of category theory, which have emerged thanks to the contribution of various authors.\\

We begin with a  review of  the classical M$\ddot{\mbox{o}}$bius theory where the main character is the set of the positive natural
numbers $\mathbb{N}_+$, which for our purposes can be regarded either as a poset or as a monoid.
We adopt first the poset viewpoint. It leads to the construction of M$\ddot{\mbox{o}}$bius theories for locally finite posets (Rota), for locally finite directed graphs, for locally finite categories (Haigh, Leinster), and for essentially locally finite categories.  Then we adopt the monoid viewpoint, which leads
to the development of  M$\ddot{\mbox{o}}$bius theories for finite decomposition monoids, for
M$\ddot{\mbox{o}}$bius categories (Haigh, Leroux), and for essentially finite decomposition categories.\\

Once we have developed M$\ddot{\mbox{o}}$bius theories for a variety of mathematical objects, a couple of natural question arise. How the various theories relate to each other? We will address this question as we develop the various theories along the notes. As a rule these relations are quite subtle a non  fully functorial. For example, one may wonder how the relation between the category of posets and the category of finite decomposition monoids may be characterized.\\

The second question asks for the common features among M$\ddot{\mbox{o}}$bius theories. In the examples we have found the following structural features:
\begin{itemize}
\item The theory depends on a fixed ring $R$ and applies to some category of objects $C$. For simplicity in these notes we consider the characteristic zero case. In a fully categorified approach one may even want to let $R$ be a ring-like category.

\item There is a construction that associates a $R$-algebra $A_c$ to each object $c \in C$,  called the incidence or convolution algebra of $c$. The correspondence $c \mapsto A_c$ need not be functorial. However, it is functorial under isomorphism, i.e. we have a functor $$A: C_{\mathrm{g}} \longrightarrow R\mbox{-}\mathrm{alg}$$ from the underlying groupoid of $C$ to the category of $R$-algebras. Often the category $C$ comes equipped with a notion of weak equivalences, which are a suitable family of morphisms between objects of $C$. The construction may fail to be functorial under equivalences even if it is functorial under isomorphism.

\item The category $C$ is monoidal with product $\times: C \times C \longrightarrow C$. The category of $R$-algebras is monoidal with product the (suitable completed) tensor product $\otimes$. The functor $$A: C_{\mathrm{g}} \longrightarrow R\mbox{-}\mathrm{alg}$$ is monoidal, i.e. there are natural isomorphisms
    $$A_{c\times d} \ \simeq \ A_c \otimes A_d.$$

\item There is a simple characterization of the units (invertible elements) in $A_c$. Moreover, there is a terminating recursive procedure that finds the inverse $u^{-1}$ for each unit  $u$ of $A_c.$ There are some distinguished unit elements, say $\xi_c \in A_c$, usually with a straightforward definition, such that their inverses $\mu_c$ are unexpectedly useful having interesting topological and geometrical properties.

   \item Often each algebra $A_c$ comes equipped  with a natural $R$-module $M_c$, leading to the M$\ddot{\mbox{o}}$bius inversion formula: $$a=b\xi_c \ \ \ \mbox{if and only if} \ \ \ b=a\mu_c  \ \ \ \mbox{for all} \ \ \ a, b \in M_c.$$

\end{itemize}

We  proceed to develop our examples of M$\ddot{\mbox{o}}$bius theories, providing along the notes references to some of the main contributions in the field, and discussing the relation between the various theories. We focus on the topological approach to understand the meaning of the M$\ddot{\mbox{o}}$bius functions, in particular, the Euler characteristics for simplicial groupoids will be useful for us. We emphasize that there is not a unique categorical generalization for the M$\ddot{\mbox{o}}$bius functions, but rather several approaches, each with its on advantages and applications.  Although we explore various routes, this work is not meant to be an exhaustive study. That would be a task exceeding the limits of these short notes. For example, we do not cover the advances developed by Cartier and Foata \cite{cart},
D$\ddot{\mbox{u}}$r \cite{dur}, Fiore, L$\ddot{\mbox{u}}$ck and Sauer \cite{fiore}, among others.  \\

This notes were prepared for the CIMPA school "Modern Methods in Combinatorics," San Luis, Argentina 2013. For the reader convenience we include a few exercises.

\section{Classical M$\ddot{\mbox{o}}$bius Theory}

In this section we review the basic elements of the classical M$\ddot{\mbox{o}}$bius theory \cite{apostol}.  Let
$\mathbb{N}_+$ be the set of positive natural numbers and $R$ be a commutative ring with identity.

\begin{defn}{\em
The M$\ddot{\mbox{o}}$bius function is the map $\mu: \mathbb{N}_+ \longrightarrow R$ given by
$$\mu(n)=
\left\{\begin{array}{cc}

1 & \ \mathrm{if} \ n =1, \ \ \ \ \ \ \ \ \ \ \ \ \ \ \ \ \ \ \ \ \ \ \ \ \ \ \ \ \ \ \ \ \ \ \ \ \ \ \ \ \ \ \ \ \ \ \ \ \ \ \ \ \ \  \\

0 & \  \mathrm{if} \ n \ \mathrm{has \ repeated \ prime \ factors,}   \ \ \ \ \  \ \ \ \ \  \ \ \ \ \ \ \ \ \ \ \ \ \ \ \\

(-1)^k & \ \mathrm{if} \ n = p_1....p_k,\ \mathrm{with} \  p_j  \mathrm{\ distinct\ prime \ numbers}.\ \ \ \

\end{array}\right.$$}
\end{defn}

Let $[\mathbb{N}_+,R]$ be the set of maps from $\mathbb{N}_+$  to $R$. We write $d|n$ if $d$ divides $n$.

\begin{defn}{\em
The Dirichlet product $\star$ on $[\mathbb{N}_+,R]$ is given on maps $f,g \in [\mathbb{N}_+,R] $ by:
$$f\star g(n) \ = \ \sum_{d|n}f(d)g(n/d)\ = \ \sum_{d|n}f(n/d)g(d)=\sum_{cd=n}f(c)g(d).$$
The unit for $\star$ is the map $1 : \mathbb{N}_+ \longrightarrow \mathbb{Z}$ given by $1(1)=1\ $ and $\ 1(n)=0\ $ for $\ n \geq 2.$}
\end{defn}

The product $\star$ turns $[\mathbb{N}_+,R]$ into a commutative $R$-algebra. Next result is known as the M$\ddot{\mbox{o}}$bius inversion formula.

\begin{prop}{\em \textbf{(M$\ddot{\mbox{o}}$bius Inversion Formula)} \\
Let $\xi \in [\mathbb{N}_+,R]$ be the map constantly equal to $1$.
The M$\ddot{\mbox{o}}$bius function $\mu \in [\mathbb{N}_+,R]$ is the $\star$-inverse of $\xi.$
Thus for $f \in [\mathbb{N}_+,R] $, we have that $$\ g=f \star \xi  \ \ \ \ \mbox{if and only if } \ \ \ \ f =  g \star \mu.$$ Equivalently,
$$g(n)\ = \ \sum_{d|n}f(d) \ \ \ \ \ \mbox{if and only if}\ \ \ \ \  f(n)\ =\ \sum_{d|n}\mu(n/d)g(d) .$$
}
\end{prop}

\begin{proof}
We show that $\xi \star \mu=1$. Since $$\xi \star \mu(1) \ = \ \xi(1) \mu(1)\ =\ 1,$$ it only remains to check that $\xi \star \mu(n)=0 \ $ for $\ n \geq 2$.
Given $n=p_1^{a_1}...p_k^{a_k}$, with $p_j$ distinct prime numbers, we have  that
$$\xi \star \mu(p_1^{a_1}...p_k^{a_k})\ = \ \sum_{d|p_1^{a_1}...p_k^{a_k}}\mu(d)\ = \ 1 \ + \ \sum_{l=1}^k\sum_{A \subseteq [k], \ |A|=l }\mu\left(\prod_{i \in A}p_i\right)\ =  $$ $$1 \ + \  \sum_{l=1}^k\sum_{A \subseteq [k], \ |A|=l }(-1)^l \ = \ \sum_{l=0}^k(-1)^l{k \choose l}\ =\ (1-1)^k\ = \ 0.$$
\end{proof}

\begin{exc}{\em   Show that $f \in  [\mathbb{N}_+,R]$ is a $\star$-unit if and only if $f(1)$ is a unit in $R$.
}
\end{exc}

\begin{exc}{\em A map $f\in [\mathbb{N}_+,R]$ is multiplicative if $f(ab)=f(a)f(b)$ for $a,b$ coprime.
Show that $1, \xi$ and $\mu$ are multiplicative functions. Show that $f\star g $ is multiplicative if $f$ and $g$ are multiplicative.
Show that the $\star$-inverse of a multiplicative function is multiplicative.
}
\end{exc}

Rooted in number theory the M$\ddot{\mbox{o}}$bius function also plays an important role in complex analysis trough the theory of Dirichlet series.

\begin{defn}{\em A Dirichlet series with coefficients in  $R$ is a formal power series
$$\sum_{n=1}^{\infty}\frac{a_n}{n^s}, \ \ \ \ \ \mbox{where} \ \  a_n \in R.$$ We let $\mathbb{D}_R$ be the $R$-algebra of Dirichlet series
with the $R$-linear product given by $$\frac{1}{n^s}\frac{1}{m^s}\ =\ \frac{1}{(nm)^s}.$$
}
\end{defn}

For our present purposes we may regard $s$  as a formal variable. In most applications one usually works over the complex numbers $\mathbb{C}$, the variable $s$ is $\mathbb{C}$-valued,  and convergency issues are of the greatest importance. The product of convergent Dirichlet series is just the product of complex valued functions.\\

The Dirichlet algebra $\mathbb{D}_R$ is isomorphic to the $R$-algebra $R<<x_1,...,x_n,...>>$ of formal $R$-linear combinations in the variables $x_n,$ with  the $R$-linear product given by
$$x_nx_m=x_{nm}, \ \  \mbox{ via the map sending} \ \  x_n \ \  \mbox{to} \ \  \frac{1}{n^s}.$$
Note that if $n=p_1^{a_1}...p_k^{a_k}$, then we have that $x_n = x_{p_1}^{a_1}...x_{p_k}^{a_k}$. Therefore $\mathbb{D}_R$ is isomorphic to  the algebra $R[[x_2,x_3, x_5,...]]$ of formal power series with coefficients in $R$ in the variables $x_p$, with $p$ a prime number.\\

\begin{prop}\label{pds}{\em The map $D:[\mathbb{N}_+,R] \longrightarrow \mathbb{D}_R$ sending $f\in [\mathbb{N}_+,R]$ to the Dirichlet series
$$D_f\ = \ \sum_{n=1}^{\infty}\frac{f(n)}{n^s}\ \ \ \ \mbox{is a ring isomorphism.}$$
}
\end{prop}

For example, the Dirichlet series $D_{\xi}$  associated with $\xi$ is the Riemann zeta function
$$\zeta(s) \ = \  \sum_{n=1}^{\infty}\frac{1}{n^s}, \ \ \ \ \mbox{and we have that} \ \ \ \zeta^{-1}(s) \ =\  \sum_{n=1}^{\infty}\frac{\mu(n)}{n^s}.$$

\begin{exc}{\em Describe explicitly the product of Dirichlet series and prove Proposition \ref{pds}.
}
\end{exc}

\begin{exc}{\em For $p \geq 2$ a primer number, let $f_p \in [\mathbb{N}_+,\mathbb{Z}]$ be such that $f_p(n)=1$ if $n=p^l$ for some $l \in \mathbb{N}$, and  $f_p(n)=0$ otherwise.
Find $f_p^{-1}, D_{f_p}$ and $D_{f_p}^{-1}.$
}
\end{exc}

\begin{exc}{\em Let $\eta \in [\mathbb{N}_+,\mathbb{Z}]$ be such that $\eta(1)=1$, $\eta(p)=-1$ for $p$ a prime number, and $\eta(n)=0$ otherwise.
Find $\eta^{-1}, D_{\eta}$ and $D_{\eta}^{-1}.$
}
\end{exc}

We close this section with a couple of facts (kind of beyond the main topic of these notes) that  show the relevance of the  M$\ddot{\mbox{o}}$bius function in number theory and complex analysis. The first fact is a reformulation in terms of the M$\ddot{\mbox{o}}$bius function of a main theorem in number theory on the asymptotic behaviour of prime numbers \cite{apostol}.

\begin{thm}{\em (\textbf{Prime Number Theorem and the M$\ddot{\mbox{o}}$bius Function})\\
Let $\pi: \mathbb{N}_+ \longrightarrow \mathbb{N}$ be such that $\pi(n)$ counts the prime numbers less than or equal to $n$.
We have that:
$$ \lim_{n \to \infty} \frac{\pi(n)ln(n)}{n}\ =\ 1, \ \ \ \mbox{or equivalently,} \ \ \ \lim_{n \to \infty}\frac{1}{n}\sum_{k \leq n}\mu(k)\ = \ 0.$$}
\end{thm}

The second fact relates the M$\ddot{\mbox{o}}$bius function to the Riemann Hypothesis \cite{sarnak, tao}. The Riemann zeta function $\zeta(s)$ can be analytically extended to the complex plane with a pole at $s=1.$ The  $\zeta$ function has zeroes
at $s=-2,-4,-6,...$, the so called trivial zeroes. Any other zero is called nontrivial, and the Riemann Hypothesis claims that all of them have real part equal to $\frac{1}{2}.$

\begin{thm}{\em (\textbf{Riemann Hypothesis and the M$\ddot{\mbox{o}}$bius Function})\\
The real part of the non-trivial zeroes of the Riemann zeta function is $\frac{1}{2}$ if and only if for any $\epsilon >0$ there exists $M >0$ such that
$$\sum_{k \leq n}\mu(k)\ < \ Mn^{\frac{1}{2}+ \epsilon }.$$
}
\end{thm}

\section{Locally Finite Posets}

In the 60's a comprehensive M$\ddot{\mbox{o}}$bius theory for locally finite partially ordered sets (posets)
was developed by Gian-Carlo Rota and his collaborators, among them Crapo, Stanley, Schmitt. Posets  are actually a particular kind of categories, so this early development may already be regarded as a categorification of the classical M$\ddot{\mbox{o}}$bius theory. For a classical introduction to categories the reader may consult Mac Lanes's book \cite{mac}. Another, quite readable, introduction to the subject is provided by    Lawvere and Schanuel \cite{ls}.
\\

Almost all the results in this section can be found in the first chapter of \cite{ro}, where the reader will find a nice mixture of original and review papers. The newer results
in this section are a discrete analogue for the Gauss-Bonnet theorem for finite posets, and the Leinster's characterization of the relation between  M$\ddot{\mbox{o}}$bius  and matrix inversion for finite
posets \cite{le,le2}. \\

A partial order on a set $X$ is a reflexive, antisymmetric, and transitive relation on $X$.
A  poset (partially ordered set) is a pair $(X, \leq)$ where $X$ is a set, and $\leq$ is a partial order on $X$. Morphism between posets are order preserving maps.
For $x \leq y$ in $X$, the  interval $[x,y]$ is the subset of $X$ given by
$$[x,y]=\{z \in X \ | \ x \leq z \leq y \} .$$ We let $$\mathbb{I}_X=\{[x,y]\ | \ x,y \in X, \ \ x \leq y \}$$ be the set of all intervals of $X$.

\begin{defn}{\em A poset is locally finite if its intervals are finite sets.
}
\end{defn}

We need a few algebraic notions such as coalgebras and Hopf algebras. The reader may consult Kassel's book  $\cite{kass}$  for a comprehensive introduction to these notions in the context quantum groups, and Cartier's notes \cite{cart1} for an historical introduction highlighting the topological connection.

\begin{lem}{\em Let  $X$ be a locally finite poset. The free $R$-module $<\mathbb{I}_X>$ generated by $\mathbb{I}_X$ together with the $R$-linear maps
$$\Delta: <\mathbb{I}_X> \ \longrightarrow \ <\mathbb{I}_X> \otimes <\mathbb{I}_X> \ \ \ \ \textrm{and} \ \ \ \ \epsilon: <\mathbb{I}_X> \ \longrightarrow \ R$$ given on generators, respectively, by
$$\Delta[x,z] \ = \ \sum_{x\leq y \leq z}[x,y]\otimes [y,z] \ \ \ \ \ \mbox{and} \ \ \ \ \
\epsilon [x,y] \ = \ \left\{\begin{array}{cc}
1 & \ \mathrm{if} \ x=y,  \\

0 & \  \mathrm{if} \ x \neq y,
\end{array}\right.$$
is a $R$-coalgebra.}
\end{lem}

\begin{proof}
The counit property is clear, indeed we have that
$$(\epsilon \otimes 1)\Delta [x,y]\ = \ \epsilon[x,x] [x,y] \ = \ [x,y]
\ = \ [x,y]\epsilon [y,y] \ = \ (1 \otimes \epsilon)\Delta[x,y].$$  Coassociativity follows from the identities
$$ (\Delta \otimes 1)\Delta[x,w]\ = \ \sum_{x\leq y \leq z \leq w}[x,y]\otimes [y,z]\otimes [z,w]\ = \ (1\otimes \Delta)\Delta[x,w].$$
\end{proof}

Recall that if $C$ is a $R$-coalgebra, then the dual $R$-module $$C^* =  \mathrm{Hom}_R(C,R)$$ is an $R$-algebra with  product $\star :C^* \otimes C^* \longrightarrow C^* $ given
on $f,g \in C^*$ by
$$f\star g(c)\ = \ (f\otimes g)\Delta(c) .$$ The $\star$-unit of $C^*$ is the counit
 $\epsilon$ of $C$.\\

For the coalgebra $<\mathbb{I}_X>$ of intervals, the dual algebra $<\mathbb{I}_X>^*$ may be identified with the algebra $[\mathbb{I}_X, R]$ of maps from $\mathbb{I}_X$ to $R$ with the product given by
$$(f\star g)[x,z]\ = \ \sum_{x\leq y \leq z}f[x,y] g[y,z] .$$
The $R$-algebra $([\mathbb{I}_X, R], \star)$ is called the incidence algebra of $X$.\\

\begin{thm}\label{l1}{\em Let $(X, \leq)$ be a locally finite poset. A map $f\in [\mathbb{I}_X, R]$  is invertible in the incidence algebra if and only if $f[x,x]$ is a unit for all $x \in X.$  In particular,
the incidence map $\xi \in [\mathbb{I}_X, R]$ constantly equal to $1$ is invertible;  its inverse $\mu \in [\mathbb{I}_X, R]$ is called the
M$\ddot{\mbox{o}}$bius function of the poset $(X, \leq).$}
\end{thm}

\begin{proof}

If $f\star g = \epsilon,$ then $f[x,x]g[x,x]=1$ and thus necessarily $f[x,x]$ is a unit for all $x \in X.$
Conversely, assume that $f[x,x]$ is a unit for $x\in X$, then its inverse $g$, if it exists,
must satisfy $g[x,x]=f[x,x]^{-1}\ $ for $\ x\in X$, and for $x< z\ $ in $X$ we have that
$$f[x,x]g[x,z]+\sum_ {x < y \leq z}f[x,y]g[y,z]\ = \ 0,$$ or equivalently
$$g[x,z]\ = \  -\sum_ {x < y \leq z}\frac{f[x,y]g[y,z]}{f[x,x]}.$$
The equations above give a terminating recursive definition for $g$  since  $X$ is locally finite.
Solving the recursion we get for $x< y$ in $X$ that $$g[x,y]\ = \ \sum_{n\geq 1}\ \sum_{x =x_0 < x_1 < ... < x_n=y}
(-1)^n\frac{f[x_0, x_1].................f[x_{n-1}, x_n]}{f[x_0,x_0]f[x_1,x_1].......f[x_{n-1},x_{n-1}]f[x_n,x_n]}$$
\end{proof}

\begin{cor}{\em  The M$\ddot{\mbox{o}}$bius function $\mu \in [\mathbb{I}_X, R]$ of a locally finite poset $(X, \leq)$ is given by
$\mu[x,x]=1\ $ for $\ x \in X$, and for $\ x < y \ $ in $X$ it is given by
$$\mu[x,y]\ = \ \sum_{n\geq 1}(-1)^n|\{x =x_0 < x_1 < .....<x_{n-1}< x_n=y\}|.$$
}
\end{cor}

\begin{prop}\label{ww}{\em Let $(X, \leq)$ and $(Y, \leq)$ be finite posets. Give $X\times Y$ the poset structure
$$(a_1,a_2) \leq (b_1,b_2)\ \ \  \mbox{if and only if} \ \ \ a_1 \leq b_1 \ \ \mbox{and} \ \ a_2 \leq b_2.$$
There is natural bijection $\mathbb{I}_{X\times Y} \longrightarrow \mathbb{I}_{X}\times \mathbb{I}_{ Y}$ sending $[(a_1,a_2), (b_1,b_2)]$ to $([a_1,b_1 ], [a_2, b_2])$ inducing an isomorphism
$$([\mathbb{I}_{X\times Y}, R], \star)  \ \simeq \ ([\mathbb{I}_{X}, R], \star) \otimes
 ([\mathbb{I}_{ Y}, R], \star).$$
Moreover, we have that $\ \xi_{X \times Y}[(a_1,a_2),(b_1,b_2)] \ = \ \xi_X [a_1,b_1]\xi_Y[a_2,b_2] \ $ and thus
$$\mu_{X \times Y}[(a_1,a_2),(b_1,b_2)] \ = \ \mu_X [a_1,b_1]\mu_Y[a_2,b_2].$$
}
\end{prop}

\begin{rem}{\em
For infinite posets we have a natural embedding of algebras
$$([\mathbb{I}_{X}, R], \star) \otimes
 ([\mathbb{I}_{ Y}, R], \star)   \ \longrightarrow  \ ([\mathbb{I}_{X\times Y}, R], \star),$$ which
 can be promoted to an isomorphisms by using a suitable completed version of the tensor product.
}
\end{rem}

Next exercises compute the M$\ddot{\mbox{o}}$bius function for a few families of posets. The reader may consult \cite{bender} for these and more examples.

\begin{exc}\label{ejf}{\em For $n\leq m$ in $\mathbb{N}$ show that the M$\ddot{\mbox{o}}$bius function of the interval $[n,m]=\{n,...,m \}$ is given by $\mu[n,n]=1, \ \mu[n,n+1]=-1, \ \mbox{and} \ \mu[n,m]=0,$ for $m \geq n + 2.$
}
\end{exc}

\begin{exc}{\em For a finite set $X$, let $(PX, \subseteq)$ be the set of subsets of $X$ ordered by inclusion. Show that the M$\ddot{\mbox{o}}$bius function of $(PX, \subseteq)$ is given for $A \subseteq B$ by
$$\mu[A,B] \ = \  (-1)^{|A \setminus B|}$$
}
\end{exc}

\begin{exc}{\em Given a finite set $X$, let $(\mathrm{Par} X, \leq)$ be set of partitions of $X$, and for
$\pi, \sigma \in  \mathrm{Par} X$ we set $\pi \leq \sigma $ if and only
if each block of $\pi$ is included in a block of $\sigma.$ Show that the M$\ddot{\mbox{o}}$bius function of $(\mathrm{Par} X, \leq)$ is given by
$$\mu[\pi, \sigma] \ = \  (-1)^{|\pi| - |\sigma| }\prod_{b \in \sigma}(n_{b}-1)!,$$
where $n_{b}$ is the number of blocks of $\pi$ included in $b.$ }
\end{exc}

Our next results require a few notions in combinatorial algebraic topology. Kozlov in $\cite{koz}$ provides a self-contained introduction
to the subject, and in particular discusses homology theory in Chapter 3.

\begin{defn}{\em
A simplicial complex $C$ consists of a set of vertices $X$ together with a family $C$ of finite non-empty subsets of $X$, called simplices, such that:
\begin{itemize}
\item $\{x\} \in C\ $ for all $\ x\in X.$

\item If $\ \emptyset \neq A \subseteq B \ $ and $\ B \in C$, then $A \in C.$
\end{itemize}
}
\end{defn}

To a finite poset $(X, \leq)$ we associate the simplicial complex $CX$ of linearly ordered subsets of $X$. It has vertex set $X$ and is such that $c \subseteq X$ is in $CX$ if and only if
the restriction of $\leq$ to $c$ is a linear, or total, order.\\

For $n \geq -1$, we let $C_nX$ be the set of elements in $CX$ of cardinality $n+1.$
By convention  we let the empty set be the unique element of $C_{-1}X$.\\

 The geometric realization of $CX$ is the topological space  $$|CX|\ = \ \left\{ a \in [X, [0,1]] \ \ \big{|} \ \ \mbox{support}(a) \in C \ \ \ \mbox{and}\ \ \ \sum_{x \in X}a(x)=1 \right\} ,$$
where

\begin{itemize}
\item $\mbox{support}(a)  =  \{x\in X \ | \  a(x)\neq 0  \}, \ $ and $\ [0,1]$ is the unit interval in $\mathbb{R},$

\item the space $[X, [0,1]] = \prod_{x\in X}[0,1]$ is given the product topology,

 \item $|CX| \subseteq [X, [0,1]] $ is given the subspace topology.
\end{itemize}

The reduced homology groups $\widetilde{\mathrm{H}}_n|CX|\ $ of $\ |CX|, \ $ for $n \geq -1$, may be identified with the homology groups of the differential
complex $(<CX>, d)$ such that  $$<CX> \ = \ \bigoplus_{n \geq -1} <C_nX>,$$ where  $<C_nX>$ is the $\mathbb{Z}$-module generated by  $C_nX$.
The differential map $$d:<C_nX> \ \ \longrightarrow \ \ <C_{n-1}X> $$ is given on generators by
$$d\{x_0 < .... < x_n\}\ = \ \sum_{i=0}^n(-1)^i\{x_0<...< \widehat{x_i} <...<x_n\}.$$
Therefore the reduced Euler characteristic of $|CX|$ is given by
$$\widetilde{\chi}|CX|\ = \ \sum_{n\geq -1}(-1)^n|C_nX|\ =\ \sum_{n \geq 0}(-1)^n\mathrm{rank}\widetilde{\mathrm{H}}_n|CX|.$$
The left identity holds by definition, the right identity is proven in \cite{koz}. We recall that the rank of an abelian group counts the number of generators of the free part of the group.\\

The Euler characteristic of $|CX|$ is defined in terms of the homology groups of the differential complex $$(<C_{\geq 0}X>, d) \subseteq (<CX>, d),$$ where
$$<C_{\geq 0}X> \ = \ \bigoplus_{n \geq 0} <C_nX>, \ \ \ \ \mbox{as follows}$$
$$\chi|CX|\ = \ \sum_{n\geq 0}(-1)^k|C_nX|\ = \ \sum_{n \geq 0}(-1)^n\mathrm{rank}\mathrm{H}_n(|CX|).$$

The following topological results are due to P. Hall.

\begin{thm}\label{t2}{\em \textbf{(Homological Interpretation of the M$\ddot{\mbox{o}}$bius Function)}\\
The M$\ddot{\mbox{o}}$bius function $\mu \in [\mathbb{I}_X, R]$ of a locally finite poset $(X, \leq)$ is given
for $x < y \in X$ by $$\mu[x,y] \ = \ \widetilde{\chi}|C(x,y)|\ = \ \sum_{n\geq -1}(-1)^n|C_n(x,y)|\ = \ \sum_{n \geq 0}(-1)^n\mathrm{rank}\widetilde{\mathrm{H}}_n|C(x,y)|,$$
where $(x,y)$ is the interval $\{ z \ | \ x < z< y \} \subseteq X$ with the induced order.}
\end{thm}

\begin{proof}
$$\mu[x,y]\ = \ \sum_{n\geq 1}(-1)^n|\{x =x_0 < x_1 < ...< x_n=y\}|\ =$$
$$ \sum_{n\geq 1}(-1)^{n-2}|C_{n-2}(x,y)|\ = \  \sum_{n\geq -1}(-1)^n|C_n(x,y)| \ = \ \widetilde{\chi}|C(x,y)| .$$
\end{proof}

Let $X$ be a finite poset, and $\overline{X}$ be the poset obtained from $X$ by adjoining a minimum $\widehat{0}$ and a maximum $\widehat{1}$  to $X$.

\begin{cor}{\em
Let $\mu_{\overline{X}} \in [\mathbb{I}_{\overline{X}} , R]$ be the M$\ddot{\mbox{o}}$bius function of $\overline{X}$, then
$$\mu_{\overline{X}}[\widehat{0},\widehat{1}]\ = \ \widetilde{\chi}|CX|\ = \ \sum_{n \geq -1}(-1)^n \mathrm{rank}\widetilde{\mathrm{H}}_n|CX|.$$
}
\end{cor}

\begin{proof}
Follows from Theorem \ref{t2} since $(\widehat{0},\widehat{1})= X$.

\end{proof}

\begin{prop}{\em The Euler characteristic $\chi|CX|$ of a finite poset $(X, \leq)$ is given by
$$\chi|CX|\ = \ \sum_{x,y \in X}\mu[x,y].$$
}
\end{prop}

\begin{proof}
The result follows from the identity $$|C_0X|\ = \ |X| \ = \ \sum_{x\in X}\mu[x,x],$$ and the fact that for $n\geq 1$ we have the identity
$$C_{n}X \ = \ \bigsqcup_{x<y} C_{n-2}(x,y).$$
Thus
$$\chi|CX|\ = \ \sum_{n\geq 0}(-1)^n|C_nX|\ = \ |C_0X| \ + \ \sum_{n\geq 1}(-1)^n|C_nX|\ = $$
$$\sum_{x\in X}\mu[x,x]\ + \ \sum_{n\geq 1}\left(\sum_{x < y}(-1)^n|C_{n-2}(x,y)|\right)\  =$$
$$\sum_{x\in X}\mu[x,x]\ + \ \sum_{x < y}\left(\sum_{n\geq 1}(-1)^{n-2}|C_{n-2}(x,y)|\right) \ =$$
$$\sum_{x\in X}\mu[x,x]\ + \ \sum_{x < y}\mu[x,y]\ = \ \sum_{x,y \in X}\mu[x,y].$$
\end{proof}

For $a\in X$, we set $X_{\geq a} =\{x \in X \ | \ x\geq a \}$.\\

 There is a structure of  right $[\mathbb{I}_X,R]$-module on
 $[X_{\geq a},R]$  via the $R$-bilinear map $$\star: [X_{\geq a},R]\times [\mathbb{I}_X,R] \longrightarrow [X_{\geq a},R]$$ sending a pair $(f,g) \in [X_{\geq a},R]\times [\mathbb{I}_X,R]$ to the map
$f\star g \in [X_{\geq a},R]$ given by $$f\star g(y)\ = \ \sum_{a\leq x\leq y}f(x)g[x,y] .$$

\begin{thm}{\em \textbf{(M$\ddot{\mbox{o}}$bius Inversion for Locally Finite Posets)}\\
Given $a\in X$ and $f, g \in [X_{\geq a},R],$ we have that $f=g\star \xi$ if and only if $g = f\star \mu,$ or equivalently
$$f(y)\ = \ \sum_{a\leq x \leq y}g(x) \ \ \ \ \mbox{if and only if}\ \ \ \ g(y)\ = \ \sum_{a\leq x \leq y}f(x)\mu[x,y].$$
}

\end{thm}

\begin{rem}{\em If $X$ is a finite set, then $[X,R]$ is a right $[\mathbb{I}_X,R]$-module and we have for $f, g \in [X, R]$ that:
$$f(y)\ = \ \sum_{x \leq y}g(x) \ \ \ \ \mbox{if and only if}\ \ \ \ g(y)\ = \ \sum_{ x \leq y}f(x)\mu[x,y].$$
}
\end{rem}

Fix a finite poset $(X, \leq)$. The set of maps $[X\times X, R] $ is naturally an $R$-algebra (square matrices indexed by $X$) with the product
$$fg(x,z)\ = \ \sum_{y \in X}f(x,y)g(y,z).$$ There is a natural embedding $[\mathbb{I}_X, R] \longrightarrow [X\times X, R] $ of algebras sending
$f\in [\mathbb{I}_X, R]$ to the map $f: X\times X \longrightarrow R$ given by
$$f(x,y)\ = \ \left\{\begin{array}{cc}
f[x,y] & \ \mathrm{if} \ x\leq y, \ \  \\

0 & \  \mathrm{otherwise}.
\end{array}\right.$$
Thus we may regard $[\mathbb{I}_X, R]$ as a subalgebra of $[X\times X, R].$ \\

A map $f \in [X\times X, R]$ is called transitive if for  $n \geq 0$, and
$x_0, ..., x_n \in X$ we have:
$$f(x_0, x_n) =0\ \ \ \ \ \mbox{implies that} \ \ \ \ \ f(x_0,x_1)f(x_1,x_2)...f(x_{n-1}, x_n)=0.$$

Leinster has shown in \cite{le, le2} the following results.

\begin{lem}\label{l2} {\em Let $f \in [X\times X, R]$ be an invertible and transitive map, then
$$f(x,y)=0\ \ \ \ \  \mbox{implies that} \ \ \ \ \ f^{-1}(x,y)=0.$$
}
\end{lem}

\begin{thm}\label{tl1}{\em Let $f \in [\mathbb{I}_X, R] \subseteq [X\times X, R]$ be a transitive map, then $f$ is invertible in $[\mathbb{I}_X, R]$
if and only if $f$ is invertible in $[X\times X, R].$
}
\end{thm}

So far we have regarded $\mathbb{I}_X$ as a set, yet $\mathbb{I}_X$ may be naturally regarded as a full
subcategory of  the category of finite posets and increasing maps. Thus we let  $\underline{\mathbb{I}_X}$ be the set of isomorphism classes of intervals in $X$. The coproduct
and counit on $<\mathbb{I}_X>$ descent to a coproduct and counit on $<\underline{\mathbb{I}_X}>$, giving an algebra structure $\star$ to $$<\underline{\mathbb{I}_X}>^* \ \ = \ \ [\underline{\mathbb{I}_X}, R],$$ the set of maps from $\underline{\mathbb{I}_X}$ to $R$, or equivalently, the set
of maps from $\mathbb{I}_X$ to $R$ invariant under isomorphisms. We call $([\underline{\mathbb{I}_X}, R], \star)$ the reduced incidence algebra of $(X, \leq)$.

\begin{exmp}{\em Let $(\mathbb{N}_+, |)$ be the poset of positive natural numbers with the order
$$n|m \ \ \ \ \ \mbox{if and only if}\ \ \ \ \  n \ \ \mbox{divides} \ \ m.$$ The reduced
incidence algebra $([\underline{\mathbb{I}}, R], \star)$ is isomorphic to the Dirichlet algebra $\mathbb{D}_R$ via the map sending
$f\in [\underline{\mathbb{I}}, R]$ to
$$D_f =  \sum_{n \in \mathbb{N}_+} \frac{f\overline{[1,n]}}{n^s}.$$ Indeed, since $[d,n] \ \simeq \ [1,n/d]$ whenever $d|n$, we have for $f,g \in [\underline{\mathbb{I}}, R]$ that
$$D_{f\star g}\ = \ \sum_{n \in \mathbb{N}_+} \frac{f\star g\overline{[1,n]}}{n^s}\ = \ \sum_{n \in \mathbb{N}_+}\left( \sum_{d|n}f\overline{[1,d]}g\overline{[d,n]} \right)\frac{1}{n^s} \ =$$
$$\sum_{n \in \mathbb{N}_+}\left( \sum_{d|n}f\overline{[1,d]}g\overline{[1,n/d]} \right)\frac{1}{n^s}\ = \  D_fD_g.$$
}
\end{exmp}

\begin{exmp}{\em  The reduced incidence algebra of the poset $(\mathbb{N}, \leq)$ is isomorphic to $R[[x]]$ via the map sending
$f\in [\underline{\mathbb{I}}, R]$ to $$\widehat{f}=\sum_{n=0}^{\infty}f\overline{[0,n]}x^n.$$ Indeed, since $[k,n] \simeq [0, n-k]$ for $k \leq n$, we have for $f, g \in [\underline{\mathbb{I}}, R]$ that:
$$\widehat{f\star g}\ = \ \sum_{n=0}^{\infty}(f\star g)\overline{[0,n]}x^n = \sum_{n=0}^{\infty}\left( \sum_{k \leq n}f\overline{[0,k]}g\overline{[0, n-k]}\right)x^n\ = \ \widehat{f}\widehat{g}.$$
}
\end{exmp}

\begin{exmp}{\em  The reduced incidence algebra of  the poset $(\mathrm{P}_f\mathbb{N}, \subseteq)$ of finite subsets of $\mathbb{N}$ ordered by inclusion is isomorphic to the divided powers algebra $$R<<x_0,x_1,...,\frac{x_n}{n!},...>>,$$ with product given by $$\frac{x_n}{n!}\frac{x_m}{m!} = {n+m \choose n}\frac{x^{n+m}}{(n+m)!},$$ via the map sending
$f\in [\underline{\mathbb{I}}, R]$ to $$\widehat{f}=\sum_{n=0}^{\infty}f\overline{[\emptyset,[n]]}\frac{x^n}{n!}.$$ Indeed, since $[a,b]\ \simeq \ [\emptyset, b\setminus a]$, we have for $f,g\in [\underline{\mathbb{I}}, R]$ that:
$$\widehat{f\star g}\ = \ \sum_{n=0}^{\infty}(f\star g)\overline{[\emptyset,[n]]}x^n\ = \
\sum_{n=0}^{\infty}\left( \sum_{a \subseteq[n]}f\overline{[\emptyset,a]}g\overline{[\emptyset, [n]\setminus a]}\right)\frac{x^n}{n!} \ = $$
$$\sum_{n=0}^{\infty}\left( \sum_{k=0}^{n}{n \choose k}f\overline{[\emptyset,[k]]}g\overline{[\emptyset, [n-k]]}\right)\frac{x^n}{n!}\ =  \ \widehat{f}\widehat{g}.$$
}
\end{exmp}

Next we show that the M$\ddot{\mbox{o}}$bius function admits a Hopf theoretical interpretation. 
This connection has been developed by Shmitt \cite{sch} in a remarkable series of papers.
For more on Hopf algebras the reader may consult \cite{cart1, kass}.

\begin{thm}\label{t1}
{\em For a locally finite poset $(X, \leq)$ we let $R[[x_{[a,b]}]]$ be the $R$-algebra of formal power series in the variables $x_{[a,b]}$ with
$a < b$ in $X$, i.e. one variable for each interval in $X$ with different endpoints. The structural maps given below turn $R[[x_{[a,b]}]]$ into a Hopf algebra such that
for $a < b$ in $X$  we have   $$\mu_X[a,b]=Sx_{[a,b]}(1).$$
}
\end{thm}

\begin{proof}
The counit, coproduct, and antipode on $R[[x_{[a,b]}]]$ are given, respectively, on generators by
$$\epsilon 1=1 \ \ \ \ \mbox{and}\ \ \ \ \epsilon x_{[a,b]}= 0,$$
$$\Delta 1=1\otimes 1 \ \ \ \ \mbox{and} \ \ \ \   \Delta x_{[a,b]} \ = \ 1\otimes x_{[a,b]} + \sum_{a< c< b}x_{[a,c]}\otimes x_{[c,b]} + x_{[a,b]}\otimes 1 ,$$
$$S 1 =1 \ \ \ \ \mbox{and}\ \ \ \ S x_{[a,b]}\ = \ \sum_{n\geq 1}\sum_{\{a =a_0 < a_1 < ...< a_n=b\}}(-1)^nx_{[a_0, a_1]}....x_{[a_{n-1}, a_n]}.$$

\end{proof}

\begin{cor}
{\em For a locally finite poset $(X, \leq)$ we let $R[[x_{\overline{[a,b]}}]]$ be the $R$-algebra of formal power series in the variables $\ x_{\overline{[a,b]}}\ $ with $\ \overline{[a,b]} \in \underline{\mathbb{I}} \ $ and $\ a < b$, i.e. one variable for each isomorphism class of intervals in $X$ with different endpoints.  The structural maps  from Theorem \ref{t1} induce structural maps on  $R[[x_{\overline{[a,b]}}]]$ turning it into a Hopf algebra such that for $a < b$ in $X$ we have: $$\mu_X\overline{[a,b]} \ =\ Sx_{\overline{[a,b]}}(1).$$}
\end{cor}

Next we state and prove a discrete analogue of the Gauss-Bonnet theorem for finite posets,
and discuss how this construction relates to the  M$\ddot{\mbox{o}}$bius function. Similar results for graphs have been developed by Knill \cite{knil}. We recall that the Gauss-Bonnet theorem
describes the Euler characteristic of a compact smooth manifold $M$ as the integral of a top differential form on $M$. \\

We have already defined the space of $\mathbb{Z}$-linear $n$-chains on a finite poset $(X, \leq)$ as the
free $\mathbb{Z}$-module $<C_nX>$ generated by the set $C_nX$ of linearly ordered subsets of $X$ of length $n+1.$
The discrete analogue $\Omega^nX$ for the differential forms of degree $n$ is simply the dual $\mathbb{Z}$-module $<C_nX>^*$,
or equivalently, the $\mathbb{Z}$-module $[C_nX, \mathbb{Z}].$ We denote by $$ \int_c \omega$$ the natural pairing
between $c \ \in \ <C_nX> \ $ and $\ \omega \in \Omega^nX=<C_nX>^*$. We use the same notation for the trivially extended pairing
$$\int: <C_{\geq 0}X> \times \Omega X \ \longrightarrow \ \mathbb{Z}, \ \ \ \ \ \mbox{where} \ \ \ $$
$$<CX> \ = \ \bigoplus_{n \geq 0} <C_nX> \ \ \ \ \  \mbox{and} \ \ \ \ \ \Omega X \ = \ \bigoplus_{n \geq 0} \Omega^nX.$$

Let $\mathbb{M}$ be the set of all maximal linearly ordered subsets of $X$. The fundamental class of $X$ is the $\mathbb{Z}$-linear chain $[X] \ \in \ <CX> \ $  given by:
$$[X] \ = \ \sum_{m\in \mathbb{M}}m .$$ For  $c \in CX$, we set $\ \mathbb{M}_c \ = \ \{m \in \mathbb{M} \ | \ c \subseteq m \}.$
 The Euler class $e_X \in  \Omega X$ of $X$ is such that
$e_X(c)=0\ $ for $ c \notin \mathbb{M}$, and for $m \in  \mathbb{M}$ we have that
$$e_X(m)\ = \ \sum_{\emptyset \neq c \subseteq m} \frac{(-1)^{|c|+1}}{|\mathbb{M}_c|}.$$
Similarly, the reduced Euler class $\widetilde{e}_X \in \Omega X$ vanishes on non-maximal linearly ordered subsets, and is given on a maximal chain $m \in \mathbb{M}$ by
$$ \widetilde{e}_X(m) \ = \ \sum_{c \subseteq m}\frac{(-1)^{|c|+1}}{|\mathbb{M}_c|}.$$

\begin{thm}{\em \textbf{(Gauss-Bonnet for Finite Posets)}\\ Let $(X, \leq)$ be a finite poset  and $|CX|$ be its geometric realization. We have that
$$\chi|CX| \ = \ \int_{[X]} e_X , \ \ \ \ \ \ \mbox{and} \ \ \ \ \ \ \widetilde{\chi}|CX| \ = \ \int_{[X]} \widetilde{e}_X .$$
}
\end{thm}

\begin{proof} We show the latter identity, the proof of the former being analogous:
$$\widetilde{\chi}|CX| \ = \ \sum_{n\geq -1}(-1)^n|C_nX| \ = \ \sum_{c \in CX}(-1)^{|c|+1} \ = \
\sum_{ c \in CX}\sum_{m \in \mathbb{M}_c}\frac{(-1)^{|c|+1}}{|\mathbb{M}_c|} \ =$$
$$\sum_{m \in \mathbb{M}}\sum_{c \subseteq m} \frac{(-1)^{|c|+1}}{|\mathbb{M}_c|}\ = \ \sum_{m \in \mathbb{M}} \widetilde{e}_X(m)\ = \ \int_{[X]} \widetilde{e}_X.$$
\end{proof}

\begin{cor}{\em  Let $(X, \leq)$ be a  poset and $x < y$ in $X,$ then $$\mu[x,y]\ = \ \widetilde{\chi}|C(x,y)|\ = \ \int_{[(x,y)]} \widetilde{e}_{(x,y)},$$
where $(x,y)$ is the poset $(x,y)=\{z \ | \ x < z < y \} \subseteq X$ with the induced order.
}
\end{cor}

\section{Locally Finite Reflexive Directed Graphs}

We proceed to consider our second example of a M$\ddot{\mbox{o}}$bius theory. Recall that  a directed graph is a pair $(X,E)$ of sets together with a (source, target) map $$(s,t): E \longrightarrow X\times X.$$
For $x, y \in X$ we set $$E(x,y)\ = \ \{ e \in E \ | \ (s,t)(e)=(x,y) \}.$$ A directed graph $(X,E)$ is called reflexive if $E(x,x) \neq \emptyset \ $ for all
$x \in X.$ \\

A walk $\gamma$ of length $l(\gamma)=n \in \mathbb{N}_+$ in $(X,E)$ is a sequence $\gamma =(\gamma_1,\gamma_2,...,\gamma_n) \in E^n$ such that
$t(\gamma_i)=s(\gamma_{i+1})$ for $i \in [n-1],$ and $s(\gamma_i)\neq t(\gamma_{i})$ for $i \in [n].$   We say that the walk $\gamma$ begins at $s(\gamma_1)$ and ends at $t(\gamma_n).$\\

For $x, y \in X$, we let $W(x,y)\ ( W_k(x,y))$ be the set of walks (of length $k$) from $x$ to $y$. \\

A circuit $\gamma$ of length $n \in \mathbb{N}_{\geq 2}$ in $(X,E)$ is a walk of length $n$ such that
$t(\gamma_n)=s(\gamma_{1})$ and $t(\gamma_i) \neq t(\gamma_j)$ for $i \neq j.$  \\

\begin{defn}{\em A graph $(X,E)$ is locally finite if $E(x,x)$ and $W(x,y)$ are finite sets for all $x, y \in X.$
}
\end{defn}

Note that a locally finite directed graph has no circuits. Moreover, a finite directed graph is locally finite if and only if it has not circuits.\\

\begin{defn}\label{voy}{\em
Given a   locally finite reflexive directed graph $(X,E)$  we let $\leq$ be the relation on $X$ such that $x\leq y$ if and only if
$x=y$ or there is walk in  $(X,E)$ from $x$ to $y$.}
\end{defn}

\begin{prop}{\em For a locally finite reflexive directed graph $(X,E)$ the pair $(X, \leq)$, with the relation $\leq$ on $X$ from
Definition \ref{voy}, is a locally finite poset.}
\end{prop}

\begin{proof}
Reflexivity is immediate. Transitivity follows since the concatenation of walks is a walk. Anti-symmetry is a consequence of the fact that $(X,E)$, a locally finite graph, has no circuits. Thus $(X, \leq)$  is a poset. It is locally finite since we have an injective map that associates to each $z$ with $x < z < y$ the walk $x \longrightarrow z \longrightarrow y$ from $x$ to $y$. Since $W(x,y)$ is a finite set, then necessarily the interval $[x,y]$ is a finite set as well.

\end{proof}

Thus  for any  locally finite reflexive directed graph $(X,E)$ we have the incidence algebra $[\mathbb{I}_{(X,\leq)}, R].$ The incidence or adjacency map $\xi \in [X\times X, R]\ $ of $\ (X,E)$ is given by
$$\xi(x,y) \ = \ |E(x,y)|.$$ Clearly, we may regard $\xi$ as an element of the incidence algebra $[\mathbb{I}_{(X,\leq)}, R].$ Note that the adjacency map $\xi_{(X,E)}$ of the graph $(X, E)$ in general is not equal to the adjacency map $\xi_{(X,\leq)}$ of the associated poset $(X, \leq)$ .\\

\begin{thm}\label{va1}
{\em \textbf{(M$\ddot{\mbox{o}}$bius Function for Locally Finite Reflexive Directed Graphs)}\\
 Let $(X,E)$ be a locally finite reflexive directed graph. The adjacency map $\xi$ of $(X,E)$ is
invertible in $[\mathbb{I}_{(X,\leq)}, R];$  its inverse $\mu$, called the M$\ddot{\mbox{o}}$bius function of $(X,E)$,  is such that $$\mu[x,x] \ = \ \frac{1}{|E(x,x)|}$$ and for $\ x \neq y$ in $X$ we have:
$$\mu[x,y]\ = \
\sum_{\gamma \in W(x,y)}\frac{(-1)^{l(\gamma)}}{|E(s(\gamma_{1}),s(\gamma_{1}))||E(t(\gamma_{1}),t(\gamma_{1})|......|E(t(\gamma_{l(n)}),t(\gamma_{l(n)})|} .$$
}
\end{thm}

\begin{proof}
Follows from  Theorem \ref{l1} the considerations above.
\end{proof}

Recall that given $a\in X$ there is a structure of  right $[\mathbb{I}_{(X, \leq)},R]$-module on
 $[X_{\geq a},R]$  via the $R$-bilinear map $\star: [X_{\geq a},R]\times [\mathbb{I}_X,R] \longrightarrow [X_{\geq a},R]$ given by $$f\star g(y)\ = \ \sum_{a\leq x\leq y}f(x)g[x,y] .$$

\begin{cor}\label{t78}{\em \textbf{(M$\ddot{\mbox{o}}$bius Inversion for Locally Finite Reflexive Directed Graphs)}\\
Fix $a \in X.$ For $f,g \in [X_{\geq a},R]$ we have that

$$g(y) \ = \ \sum_{e \in E, \ a \leq se , \ te =y} f(se) \ \ \ \ \mbox{if and only if}\ \ \ \ f(y)\ = \ \sum_{a \leq x \leq y} g(x)\mu(x,y).$$}
\end{cor}

\begin{proof}We have that $g=f\star \xi \ $ if and only if $\ g=g\star \mu. $ The result follows since
$$ g(y)\ = \ f\star \xi(y)\ = \ \sum_{a\leq x\leq y}f(x)\xi[x,y]\ = \ \sum_{e \in E, \ a \leq se , \ te =y} f(se).$$
\end{proof}

Given directed graphs $(X, E)$ and $(Y, F)$, the product graph $(X\times Y, E\times F)$ is such that
an edge from $(a_1,b_1)$ to $(a_2,b_2)$ in $X\times Y$ is the same as a pair of edges $(e,f) \in E\times F$ where $e$ is an edge from $a_1$ to $a_2$ in $X$, and $f$ is an edge from $b_1$ to $b_2$ in $Y.$

\begin{prop}{\em Let $(X, E)$ and $(Y, F)$ be finite reflexive circuit-less directed graphs, the
the product graph $(X\times Y, E\times F)$ is also a finite reflexive circuit-less, and we have an isomorphism of posets $$(X\times Y, \leq) \ \simeq \ (X,\leq) \times (Y,\leq) $$
thus we have a natural isomorphism of algebras
$$([\mathbb{I}_{(X\times Y, \leq)}, R], \star)  \ \simeq \ ([\mathbb{I}_{(X,\leq)}, R], \star) \otimes
 ([\mathbb{I}_{(Y,\leq)}, R], \star).$$
Moreover, we have that $$ \xi_{X \times Y}[(a_1,b_1),(a_2,b_2)] \ = \ \xi_X [a_1,a_2]\xi_Y[b_1,b_2]$$  and thus
$$\mu_{X \times Y}[(a_1,b_1),(a_2,b_2)] \ = \ \mu_X [a_1,a_1]\mu_Y[b_1,b_2].$$
}
\end{prop}

\

Let $(X, \leq)$ be a locally finite poset and consider the cover relation $\prec$ on $X$ given by
$$x \prec y \ \ \mbox{if and only if} \ \ x < y \ \ \mbox{and there is no} \ \
z \in X \ \  \mbox{such that} \ \  x < z <y.$$

Consider the map $\eta \in [\mathbb{I}_{(X, \leq)}, R]$ given by
$$\eta[x,y] \ = \
\left\{\begin{array}{cc} 1 & \ \mathrm{if} \ x = y, \\
-1 & \ \ \  \mathrm{otherwise.}  \end{array}\right.
$$
It follows from Theorem \ref{va1} that $\eta \in [\mathbb{I}_{(X, \leq)}, R]$ is invertible and its inverse $\eta^{-1} \in [\mathbb{I}_{(X, \leq)}, R]$ is such that
$$\eta^{-1}[x,x] \ = \ 1\ \ \ \ \ \mbox{and}\ \ \ \ \ \eta^{-1}[x,y] \ = \ |\mathbb{M}[x,y]|,$$
where $\mathbb{M}[x,y]$ is the set of maximal linearly ordered subsets of the interval $[x,y] \subseteq X.$ \\

Fix $a \in X$,  the finite difference operator $$\Delta: [X_{\geq a}, R] \longrightarrow [X_{\geq a}, R] $$ is given for $f \in [X, R]\ $ and $\ y \in X_{\geq a}$ by
$$\Delta f (y) \ = \ f(y)\ - \ \sum_{a \leq x \prec y}f(x) \ = \ f\star \eta (y). $$
The M$\ddot{\mbox{o}}$bius inversion formula tell us that
$$\Delta f \ = \ g  \ \ \ \ \ \mbox{if and only if} \ \ \ \ \ f(y)\ = \ \sum_{a \leq x \leq y}|\mathrm{M}[x,y]|g(x).$$

\section{Locally Finite Categories}\label{s1}

In this section we introduce the coarse M$\ddot{\mbox{o}}$bius  theory for categories (although using a different terminology and focusing on a concrete case) developed
by Leinster \cite{le, le2}. This sort of M$\ddot{\mbox{o}}$bius theory depends only on the underlying graph of the category. \\

Given a category $C$ we let $C_0$ be the collection of its  objects. Abusing notation we usually write $x \in C$ instead of $x \in C_0$.
Let $C_1$ be the collection of morphisms in $C$, and  $C(x,y)$  be the set of morphisms in $C $ from $x$ to $y$. The source and target maps
$$(s,t): C_1 \longrightarrow C_0 \times C_0$$ together with the  map $1:C_0 \longrightarrow C_1$, sending each object $x\in C$ to its identity $1_x \in C(x,x), \ $ give $C$ the structure of
a reflexive directed graph. Of course, in a category we have in addition the composition maps
$$\circ :  C(x,y) \times C(y,z) \longrightarrow C(x,z)$$ sending a pair of morphisms $(f,g) \in  C(x,y) \times C(y,z)$
to its composition $gf.$ Composition of morphisms is associative and unital in the sense that
$$(hg)f = h(gf) \ \ \ \mbox{and}\ \ \ 1_yf=f=f1_x, \ \ \ \mbox{for} \ \ \  f,g,h \in  C(x,y) \times C(y,z) \times C(y,z) .$$
The notation $x {\overset{f} \longrightarrow} y$ means that $f$ is a morphism in $C(x,y)$, and
$x = sf$ and  $y=tf$. For $x,y \in C$ we set $$[x,y]\ =\ \{z \in C_0 \ | \ \mbox{there is a diagram} \ \ x {\longrightarrow} z {\longrightarrow}y\ \ \mbox{in} \  \ C \ \}.$$

\begin{defn}\label{loc}
{\em A category $C$ is locally finite if the following conditions hold:
\begin{itemize}
\item $C(x,y)$ is a finite set for all $x,y \in C.$
\item If there is a diagram $x {\longrightarrow} y {\longrightarrow}x$ in $C$, then $x=y.$
\item For $x,y \in C$, the set $[x,y]$ is a finite.
\end{itemize}
}
\end{defn}

\begin{rem}{\em
Locally finite categories are skeletal.
}
\end{rem}

\begin{exmp}{\em
A finite category (i.e. a category with a finite set of morphisms) is locally finite if and only if in any diagram
$x {\longrightarrow} y {\longrightarrow}x$, we have that $x=y$.
}
\end{exmp}

\begin{exmp}{\em
Let $C$ be a category with $C(x,y)$ finite for all $x,y \in C,$ let $(X, \leq )$ be a locally finite poset, and let $F:X \longrightarrow C$ be a functor.
The category $X_F$ with objects $X$ and morphisms given by
$$X_F(x,y) \ = \
\left\{\begin{array}{cc} C(F(x),F(y)) & \ \mathrm{if} \ x \leq y, \\
\emptyset & \ \ \  \mathrm{otherwise,}  \end{array}\right.
$$ is locally finite.
}
\end{exmp}

\begin{prop}{\em A category is locally finite if and only if its  graph is locally finite.
}
\end{prop}

\begin{proof}
Let $C$ be a locally finite category. Then $C(x,x)$ is a finite set for all $x \in C.$ For $x\neq y \in C$, consider the set $W(x,y)$ of walks from $x$ to $y$.
The objects in a walk
$$x {\longrightarrow} x_1 {\longrightarrow}  .....\longrightarrow x_{n}\longrightarrow y $$
are all distinct because in any configuration of the form
$$x {\longrightarrow} x_1 {\longrightarrow}  .....\longrightarrow x_{k}\longrightarrow x $$
we must have that $x=x_1=...=x_k$. Notice that if an object $z \in C$ appears in a walk from $x$ to $y$, then $z\in [x,y].$
Since $[x,y]$ is a finite set, then necessarily $W(x,y)$ is also a finite set. In particular we have that $W(x,x)=\emptyset.$ Thus the graph of $C$ is locally finite. \\

Conversely, assume that the graph of $C$ is locally finite. The sets $C(x,x)$ are finite by definition. For $x \neq y,$ the sets $C(x,y)$ and $[x,y]$ are finite,
since there is an injective map $C(x,y) \longrightarrow W(x,y)$, and  a surjective map
$$\{ x\} \sqcup W_2(x,y) \sqcup \{y \} \longrightarrow  [x,y] .$$ Recall that a locally finite graph has no circuits,
thus in any configuration $x {\longrightarrow} y {\longrightarrow}x$ we must have that $x=y.$ In particular $[x,x]=\{x\}.$
\end{proof}

\begin{lem}\label{order}
{\em Let $C$ be a category with $C(x,y)$ finite  for  $x,y \in C.$ Let $\leq$ be the relation on $C_0$ given by
 $x\leq y$ if and only if there is a morphism $x\longrightarrow y$ in $C$.
Then $C$ is a locally finite category if and only if  $x\leq y$ defines a locally finite partial order on $C_0.$
}
\end{lem}

\begin{proof}
Assume $C$ is a locally finite category, then $C_0$ is a poset with the ordering from statement of the theorem. Reflexivity and transitivity are immediate. Anti-symmetry follows from the second property in Definition \ref{loc}. The poset $(C_0, \leq)$ is locally finite by the third property in  Definition \ref{loc}.\\

Conversely, if  $(C_0, \leq)$ is a locally finite poset, then a diagram $x \longrightarrow y \longrightarrow x$ implies that $x=y$ by the reflexivity of $\leq.$ The interval $[x,y]$ is finite since
$$[x,y]\ =\ \{z \in C_0 \ | \ \mbox{there is a diagram} \ \ x {\longrightarrow} z {\longrightarrow}y\ \ \mbox{in} \  \ C \ \} \ = \ \{z \in C_0\ | \  x \leq z \leq y \}.$$
\end{proof}

Thus for any locally finite category $C$ we have the partially order set $(C_0,\leq)$ and the corresponding incidence algebra $([\mathbb{I}_{(C_0,\leq)}, R], \star)$.
The incidence or adjacency map of $C$ $$\xi:C_0 \times C_0  \longrightarrow R$$ is given by $$\xi[x,y]=|C(x,y)|$$ and lies naturally in $[\mathbb{I}_{(C_0,\leq)}, R].$

\begin{thm}\label{t3}{\em \textbf{(M$\ddot{\mbox{o}}$bius Function for Locally Finite Categories)}\\
Let $C$ be a locally finite category. The adjacency map $\xi$ of $C$ is
invertible in $[\mathbb{I}_{(C_0, \leq)}, R];$  its inverse $\mu$, called the M$\ddot{\mbox{o}}$bius function of $C$,
is such that $$\mu[x,x]  =  \frac{1}{|C(x,x)|}$$ and for $x \neq y \in C$ we have that
$$\mu[x,y]\ = \
\sum_{n \geq 1}\sum_{x =x_0 < x_1 < ... < x_n=y}
\frac{(-1)^n \ |C(x_0, x_1)|.................|C(x_{n-1}, x_n)|}{|C(x_0,x_0)||C(x_1,x_1)|.......|C(x_{n-1},x_{n-1})||C(x_n,x_n)|}\ = $$
$$\sum_{n \geq 1}\sum_{x =x_0 \rightarrow x_1 \rightarrow ... \rightarrow x_n=y}
\frac{(-1)^n}{|C(x_0,x_0)||C(x_1,x_1)|.......|C(x_{n-1},x_{n-1})||C(x_n,x_n)|} ,$$
where in the second identity the sum ranges over all diagrams in $C$ of the form $$x =x_0 \longrightarrow x_1 \longrightarrow ...... \longrightarrow x_n=y.$$
}
\end{thm}

\begin{proof} The result follows  from Theorem \ref{l1}.
\end{proof}

Fix $a \in C$. Then $[C_{0,{\geq a}}, R]$ is a right $[\mathbb{I}_{(C_0, \leq)}, R]$-module, thus we get the following result.

\begin{cor}\label{t3}{\em \textbf{(M$\ddot{\mbox{o}}$bius Inversion for Locally Finite Categories)}\\
Fix $a \in X.$ For $f,g \in [C_{0,{\geq a}},R]$ we have that
$$g(y)\ = \ \sum_{a \leq x \leq y} f(x)|C(x,y)| \ \ \ \ \mbox{if and only if}\ \ \ \ f(y)\ =\ \sum_{a \leq x \leq y} g(x)\mu(x,y).$$

}
\end{cor}

Recall that if $C$ and $D$, then objects in the product category $C\times D$ are pairs $(a,b)$ with $a \in C$ and $b\in D.$ Morphisms in $C\times D$ are given by
$$C\times D((a_1, b_1), (a_2,b_2)) \ = \
C(a_1, a_2) \times D(b_1,b_2) .$$

\begin{prop}{\em Let $C$ and $D$ be locally finite categories, then
the product $C\times D$ is also a locally finite category, and we have a natural isomorphism of posets
$$\mathbb{I}_{((C\times D)_0,\leq)} \ \simeq \ \mathbb{I}_{(C_0,\leq)} \times \mathbb{I}_{(D_0,\leq)},$$
and thus a natural isomorphism of algebras
$$([\mathbb{I}_{((C\times D)_0,\leq)}, R], \star)  \ \simeq \ ([\mathbb{I}_{(C_0,\leq)}, R], \star) \otimes
 ([\mathbb{I}_{(D_0,\leq)}, R], \star).$$
Moreover, we have that $$ \xi_{C \times D}[(a_1,b_1),(a_2,b_2)] \ = \ \xi_C [a_1,a_2]\xi_D[b_1,b_2]$$ and thus
$$\mu_{C \times D}[(a_1,b_1),(a_2,b_2)] \ = \ \mu_C [a_1,a_1]\mu_D[b_1,b_2].$$
}
\end{prop}

\begin{exmp}{\em Let $C$ be a locally finite category, $a \in C,$ and $F:C \longrightarrow \mathrm{set}$ be a functor from $C$ to the category of finite sets.
Consider the functor $G:C \longrightarrow \mathrm{set}$ given on objects by
$$G(y ) \ = \ \bigsqcup_{a \leq x \leq y} F(x)\times C(x,y ).$$
By the  M$\ddot{\mbox{o}}$bius inversion formula we have that
$$|F(y)|\ = \ \sum_{a \leq x \leq y} |G(x)|\mu(x,y) .$$  }
\end{exmp}

\section{Essentially Locally Finite Categories}

The  M$\ddot{\mbox{o}}$bius theory for locally finite categories developed in Section \ref{s1} although functorial under isomorphisms of categories, fails
to be functorial under equivalences of categories. We recall that categories $C$ and $D$ are equivalent if there is a functor $F:C \longrightarrow D$ that is essentially
surjective, full and faithful \cite{mac}. A category may fail to be locally finite and yet be equivalent to a locally finite category. For example, the category with two objects, a unique isomorphism between
them, and identities as the only endomorphism, is equivalent to the category with one object and one morphism. The latter is locally finite whereas the former is not.\\

Given a category $C$ and objects $x,y \in C$, the notation $x \simeq y$ means that $x$ and $y$ are isomorphic objects. Let $\overline{C}$ be set of
isomorphism classes of objets in $C$, i.e. $\overline{C}$ is the quotient set $C_0 / \simeq$. A typical element of $\overline{C}$ is denoted
by $\overline{x}$, meaning that we have an equivalence class $\overline{x} \in \overline{C}$ and that we have chosen a representative object
$x \in \overline{x}.$ For $\overline{x}, \overline{y} \in \overline{C}$ we set
$$[\overline{x}, \overline{y}] \ = \ \{\overline{z} \in \overline{C} \ | \ \mbox{there is a diagram}  \ \  x \longrightarrow z \longrightarrow y \ \ \mbox{in}\  C \ \} .$$

\begin{defn}\label{tr}
{\em A category $C$ is essentially locally finite if the following conditions hold:
\begin{itemize}
\item $C(x,y)$ is a finite set for $x,y \in C.$
\item If we have a diagram $x {\longrightarrow} y {\longrightarrow}x \ $ in $\ C$, then $x\simeq y.$
\item For $x,y \in C$, the set $[\overline{x},\overline{y}]$ is finite.
\end{itemize}
}
\end{defn}

\begin{rem}{\em In
the applications we often find a stronger version of the second property in Definition \ref{tr}: the arrows in any configuration
$x {\longrightarrow} y {\longrightarrow}x$ are isomorphisms. We call categories with such a property  isocyclic, i.e. all cycles of morphisms in $C$
are formed by isomorphisms. Not all essentially locally finite categories are isocyclic, e.g. a finite monoid regarded as a category.
}
\end{rem}

\begin{lem}\label{qorder}
{\em Let $C$ be a category with $C(x,y)$ finite  for  $x,y \in C.$ Let $\leq$ be the relation on $\overline{C}$ given by $\overline{x}\leq \overline{y} \ $ if and only if there is a morphism $x \longrightarrow y$ in $C$.
Then $C$ is an essentially  locally finite category if and only if
$\ \leq \ $ is a locally finite partial order on $\overline{C}.$
}
\end{lem}

\begin{proof} Reflexivity and transitivity of $\leq$ are obvious. Antysymmetry is equivalent to the second property of Definition \ref{tr}. The third property in Definition \ref{tr} is equivalent to local finiteness.
\end{proof}

\begin{prop}{\em A category $C$ is essentially locally finite if and only if $C$ is equivalent to a locally finite category.
}
\end{prop}

\begin{proof} Assume $C$ is an essentially locally finite category. Let $S$ be a full subcategory of $C$ whose objects include one and only one representative
of each isomorphism class of $C$. The category $S$ is equivalent to $C$ to and skeletal, thus it is a locally finite category. \\

Conversely, let $S$ be a locally finite category and  $F:S \longrightarrow C$ an equivalence of categories.
For objects $y_1, y_2 \in C$ there are objects $x_1, x_2 \in  S$ such that
$$F(x_1)\ \simeq \ y_1 \ \ \ \ \ \mbox{and} \ \ \ \ \ F(x_2)\ \simeq \ y_2.$$ Moreover, we have
bijective maps
$$C(x_1, x_2) \longrightarrow C(y_1, y_2) \ \ \ \ \ \ \mbox{and} \ \ \ \ \ \ [x_1, x_2] \longrightarrow [y_1, y_2].$$
Thus $C(y_1, y_2)$ and $[y_1, y_2]$ are finite sets. If there is a diagram $y_1 \longrightarrow y_2 \longrightarrow y_1$ in $C$, then there is a corresponding
diagram $x_1 \longrightarrow x_2 \longrightarrow x_1$ in $S$, thus $x_1 = x_2$ and $y_1 \simeq y_2.$

\end{proof}

\begin{exmp}{\em Let $C$ be a subcategory of the category  of finite sets and maps. Let $IC$ be the subcategory of $C$ with the same objects as $C$ and such that $f\in IC(x,y) \subseteq C(x,y)$ if and only if $f$ is an injective map from $x$ to $y$. Dually, let $SC$ be the subcategory of $C$ with the same objects as $C$ and such that $f\in SC(x,y) \subseteq C(x,y)$ if and only if $f$ is a surjective map. The categories
$IC$ and $SC$ are isocyclic and essentially locally finite categories.
}
\end{exmp}

Let $\mathbb{I}$ be the category of finite sets and injective maps. A combinatorial presheaf $P$ is contravariant functor
$P:\mathbb{I} \longrightarrow \mathrm{set} $. The Grothendieck category of elements $\mathbb{I}_P$   has for objects pairs $(x,a)$ where $a \in Px.$
A morphism $$(x,a) \longrightarrow (y,b)\  \in \ \mathbb{I}_P$$ is an injective map $f: x \longrightarrow y$ such that $$b|_x= P_fb=a.$$

\begin{prop}{\em
For any combinatorial presheaf $P$ the category of elements $\mathbb{I}_P$
is isocyclic and essentially locally finite.}
\end{prop}

\begin{proof}A composition of maps between finite sets is a bijection if and only if the maps are bijections. Whenever we have an element $\overline{(z,c)}$ in an interval $[\overline{(x,a)},\overline{(y,b)}]$ of $\overline{\mathbb{I}}_P$ we may assume, using isomorphic representations, that $$x \subseteq z \subseteq y, \ \ \ \ \ a=c|_x c= b|_z \ \ \ \ \ \mbox{and} \ \ \ c= b|_z .$$ Thus there is only a finite number of choices for $\overline{(z,c)}$ and $\mathbb{I}_P$ is an essentially locally finite category.
\end{proof}

Let $\mathbb{B}$ be the category of finite sets and bijections, and let $O$ be an operad in the category of finite sets. Thus $O$ is a functor $$O: \mathbb{B}_+ \longrightarrow \mathrm{set},$$
from the category $\mathbb{B}_+$ of non-empty finite sets and bijections to the category of finite sets, together with a distinguished element $1 \in O[1]$ and suitable composition maps
$$m_{\pi} : O(\pi)\times \prod_{b \in \pi}O(b) \longrightarrow O(x) ,$$ where $\pi$ is a partition  of the finite set $x$. The reader may consult \cite{re2, lo} and the references there in for details. Let $\mathbb{S}_O$ be the category whose objects
are finite sets, and such that a morphism $(f,a)\in \mathbb{S}_O(x,y)$ is a  surjective map $f:x\longrightarrow y$ together with an element $$a \in \prod_{j \in y}O(f^{-1}j).$$
Composition of morphisms is defined with the help of the operadic compositions as follows. Suppose we have morphisms
$$x \overset{(f,a)} \longrightarrow y \overset{(g,b)} \longrightarrow z \ \ \ \ \  \mbox{ where}$$
\begin{itemize}
\item $f:x \longrightarrow y \ $ is a surjective map, and $\ a=(a_j)_{j \in y} \ $ with
 $\ a_j \in O(f^{-1}j),  $

\item $g:y \longrightarrow z \ $ is a surjective map, and $\ b=(b_k)_{k \in z} \ $ with
$\ b_k \in O(g^{-1}k).$
\end{itemize}
The composition $$(g,b) \circ (f,a)= (gf, ba)$$ is such that $gf:x \longrightarrow z$ is the composition map, and for $k \in z$ the element $(ba)_k\in O((gf)^{-1}k)$ is defined as follows. Since
$$(gf)^{-1}k  \ = \  \bigsqcup_{j\in g^{-1}k}f^{-1}j\ \ \ \ \mbox{we get a map}$$
$$m : O(\{f^{-1}j\}_{j \in g^{-1}k})\times \prod_{j \in g^{-1}k}O(f^{-1}j) \longrightarrow O((gf)^{-1}k) ,$$
or equivalently a map
$$ m : O(g^{-1}k)\times \prod_{j \in g^{-1}k}O(f^{-1}j) \longrightarrow O((gf)^{-1}k).$$
We let $(ba)_k \in O((gf)^{-1}k) $ be given by
$$(ba)_k \ = \ m(b_k \ , \ (a_j)_{j \in g^{-1}k}) .$$

\begin{prop}{\em
For any operad $O$ of finite sets the category $\mathbb{S}_O$  is isocyclic and essentially locally finite.}
\end{prop}

\begin{proof} Again a composition of maps between finite sets is a bijection if and only if the maps  are bijections. Whenever we have an element $\overline{(z,c)}$ in an interval $[\overline{(x,a)},\overline{(y,b)}]$ of $\overline{\mathbb{S}}_O$ we may assume that $z$ is a partition of $x$, using isomorphic representations, thus there is only a finite number of choices for $\overline{(z,c)}$ and $\mathbb{S}_O$ is a locally finite category.
\end{proof}

For any essentially locally finite category $C$ we have the partially order set $(\overline{C},\leq)$ and the incidence algebra
$([\mathbb{I}_{(\overline{C},\leq)}, R], \star)$.
The incidence or adjacency map $\xi:\overline{C} \times \overline{C}  \longrightarrow R$ of $C$  $$\xi(\overline{x},\overline{y})=|C(x,y)|$$
lives naturally in $[\mathbb{I}_{(\overline{C},\leq)}, R].$

\begin{thm}\label{t3}{\em \textbf{(M$\ddot{\mbox{o}}$bius Function for Essentially Locally Finite Categories)}\\
Let $C$ be an essentially locally finite category. The adjacency map $\xi$ is
invertible in $[\mathbb{I}_{(\overline{C}, \leq)}, R];$  its inverse $\mu$, called the  M$\ddot{\mbox{o}}$bius function of $C$,
is such that $$\mu[\overline{x},\overline{x}] \ = \ \frac{1}{|C(x,x)|}$$ and for $\ \overline{x} \neq \overline{y} \ $ in $C$ we have that
$$\mu[\overline{x},\overline{y}]\ = \
\sum_{n \geq 1}\sum_{\overline{x} \ = \
\overline{x}_0 < \overline{x}_1 < ... < \overline{x}_n=\overline{y}}
(-1)^n \frac{\ |C(x_0, x_1)|.................|C(x_{n-1}, x_n)|}{|C(x_0,x_0)||C(x_1,x_1)|.......|C(x_{n-1},x_{n-1})||C(x_n,x_n)|}.$$
}
\end{thm}

\begin{proof}Follows from Theorem \ref{l1}.
\end{proof}

For $\overline{a} \in \overline{C}$, we have that $[\overline{C}_{\geq \overline{a}}, R]$ is a right $[\mathbb{I}_{(\overline{C}, \leq)}, R]$-module.

\begin{cor}\label{t9}{\em \textbf{(M$\ddot{\mbox{o}}$bius Inversion for Essentially Locally Finite Categories)}\\
Fix $\overline{a} \in \overline{C}$. For $f,g \in [\overline{C}_{\geq \overline{a}},R]$ we have that

$$g(\overline{y})\ = \ \sum_{\overline{a} \leq \overline{x} \leq \overline{\overline{y}}} f(\overline{x})|C(\overline{x},\overline{y})| \ \ \ \ \
\mbox{if and only if}\ \ \ \ \ f(y) \ = \ \sum_{\overline{a} \leq \overline{x} \leq \overline{y}} g(\overline{x})\mu(\overline{x},\overline{y}).$$

}
\end{cor}

\begin{prop}{\em Let $C$ and $D$ be finite esentially locally finite categories, then
 $C\times D$ is a finite essentially locally finite category, and we have a natural isomorphism of posets
$$\mathbb{I}_{(\overline{C\times D},\leq)} \ \simeq \ \mathbb{I}_{(\overline{C},\leq)} \times \mathbb{I}_{(\overline{D},\leq)}$$
thus we have a natural isomorphism of algebras
$$([\mathbb{I}_{(\overline{C\times D},\leq)}, R], \star)  \ \simeq \ ([\mathbb{I}_{(\overline{C},\leq)}, R], \star) \otimes
 ([\mathbb{I}_{(\overline{D},\leq)}, R], \star).$$
Moreover, we have that $$ \xi_{C \times D}[\overline{(a_1,b_1)},\overline{(a_2,b_2)}] \ = \ \xi_C [\overline{a}_1,\overline{a}_2]\xi_D[\overline{b}_1,\overline{b}_2]$$ and thus
$$\mu_{C \times D}[\overline{(a_1,b_1)},\overline{(a_2,b_2)}] \ = \ \mu_C [\overline{a}_1,\overline{a}_2]\mu_D[\overline{b}_1,\overline{b}_2].$$
}
\end{prop}

\

The formula for the M$\ddot{\mbox{o}}$bius function $\mu$ from Theorem \ref{t3}  admits a nice conceptual understanding in the case of isocyclic categories which we proceed
to formulate. We introduce first a few required mathematical notions. \\

The cardinality of finite sets can be viewed as an invariant under isomorphisms map $$|\ |:\mathrm{set} \longrightarrow \mathbb{N}$$ from the category of finite sets to the set of natural numbers, which satisfies $$|\emptyset|=0, \ \ \ \ \ |[1]|=1, \ \ \ \ \ |x \sqcup y|= |x| + |y|, \ \ \ \ \
 \mbox{and} \ \ \ \ \ |x\times y|=|x||y|.$$

The notion of cardinality for finite sets admits a suitable extension \cite{BaezDolan,Blan1,diaz} to the category $\mathrm{gpd}$ of essentially finite groupoids (i.e. groupoids equivalent to finite groupoids) via the invariant under equivalences map
$$| \ |_{\mathrm{g}}: \mathrm{gpd}\longrightarrow \mathbb{Q},$$   given by
$$|G|_{\mathrm{g}}\ = \ \sum_{x\in \overline{G}} \frac{1}{|G(x,x)|}.$$ We recall that a groupoid is category with all  morphisms invertible; a groupoid is essentially finite if it is equivalent to a groupoid with a finite number of morphism. The map $| \ |_{\mathrm{g}}$ is invariant under equivalence of groupoids and is such that
$$|\emptyset|=0, \ \ \ \ \ |[1]|=1, \ \ \ \ \ |G \sqcup H|= |G| + |H|, \ \ \ \ \  \mbox{and} \ \ \ \ \
|G\times H|=|G||H|.$$ Another useful property of $| \ |_{\mathrm{g}}$ is the following. Assume a finite group $G$ acts on the finite set $X$, then
we let $X\rtimes G$ be the groupoid  with set of objects $X$ and
such that $$X\rtimes G(x,y)\ = \ \{g\in G \ | \ gx=y \}.$$ It is easy to check that
$$|X\rtimes G|_{\mathrm{g}}\ = \ \frac{|X|}{|G|}.$$
Note that $\ \overline{X\rtimes G}\ $ is the  quotient set $X/G.$\\

Our immediate goal is to understand the M$\ddot{\mbox{o}}$bius function for isocyclic essentially locally finite categories in terms of the cardinality of groupoids, so we fix one of those categories $C$. For $\overline{x} \in \overline{C}$, the set of morphisms $C(x,x)$ from $x$ to itself is a group, moreover the group $C(x,x)\times C(x,x)$ acts on  $C(x,x)$ by pre and post composition of morphisms. We have that
$$\mu[\overline{x}, \overline{x}] \ = \ \frac{1}{|C(x,x)|}\ = \ \frac{|C(x,x)|}{\ |C(x,x)|^2}\ =  \ \left|C(x,x)\rtimes ({C(x,x)\times C(x,x)})\right|_{\mathrm{g}} .$$

For $\overline{x} < \overline{y}\ $ in $\overline{C}$, consider the groupoid  $\ [[0,n],C]_{\overline{x},\overline{y}}^{\natural}\ $ (the odd notation will be
justified below) whose objects are functors $$F:[0,n] \longrightarrow C $$ from the interval $[0,n] \subseteq \mathbb{N} \ $ to  $\ C$ such that $F(0) \simeq x, \ F(n) \simeq y$, and $F(i < i+1) $ is not an isomorphism for  $i \in [0,n-1]$.\\

Morphisms in $[[0,n],C]_{\overline{x},\overline{y}}^{\natural}\ $ are natural isomorphisms.
Concretely, objects in  $[[0,n],C]_{\overline{x},\overline{y}}^{\natural}\ $
are diagrams in $C$ of the form $$x \simeq x_0 \longrightarrow x_1 \longrightarrow ...... \longrightarrow x_n \simeq y,$$ where none of the arrows
 is an isomorphism. Morphisms in  $[[0,n],C]_{\overline{x},\overline{y}}^{\natural}$  are commutative diagrams
\[\xymatrix @R=.3in  @C=.4in
{x_0 \ar[r] \ar[d] &  x_1 \ar[r] \ar[d] & \ \ \ \ \ \ \ .... \ \ \ \ \ar[r] & x_{n-1}  \ar[r] \ar[d] & x_n \ar[d] \\
z_0 \ar[r]  &  z_1 \ar[r]  &  \ \ \ \ \ \ \ .... \ \ \ \ \ar[r] &  z_{n-1} \ar[r] & z_n} \]
where the vertical arrows are isomorphisms.\\

Let us compute the cardinality of the groupoid $[[0,n],C]_{\overline{x},\overline{y}}^{\natural}$. Note that
$$[[0,n],C]_{\overline{x},\overline{y}}^{\natural} \ \ = \ \
\bigsqcup_{\overline{x} \simeq \overline{x}_0 < \overline{x}_1 < ...... < \overline{x}_{n-1}< \overline{x}_n \simeq y} [[0,n],C]_{\overline{x}, \overline{x}_1, ...., \overline{x}_{n-1}, \overline{y}}$$ where the groupoids $\ [[0,n],C]_{\overline{x}, \overline{x}_1, ...., \overline{x}_{n-1}, \overline{y}} \ $ are defined just
as $\ [[0,n],C]_{\overline{x},\overline{y}}\ $ fixing beforehand the isomorphism classes of the intermediate objects $\overline{x}_1, ...., \overline{x}_{n-1}$.
The groupoids  $$[[0,n],C]_{\overline{x}, \overline{x}_1, ...., \overline{x}_{n-1}, \overline{y}}
$$ are actually quite easy to understand.
Objects and morphism in it are isomorphic to diagrams of the form
 \[\xymatrix @R=.4in  @C=1in
{x \ar[r]^{f_1} \ar[d]_{g_0} &  x_1 \ar[r]^{f_2\ \ \ \ \ } \ar[d]_{g_1} & \ \ \ \ \ \ \ .... \ \ \ \ \ar[r]^{\ \ f_{n-1}} & x_{n-1}  \ar[r]^{f_n} \ar[d]_{g_{n-1}} & x_n \ar[d]_{g_n} \\
x \ar[r]^{g_1f_1g_0^{-1}}  &  x_1 \ar[r]^{g_2f_2g_1^{-1}\ \ \ \ \ }  &  \ \ \ \ \ \ \ .... \ \ \ \ \ar[r]^{\ \ g_{n-1}f_{n-1}g_{n-2}^{-1}} &  x_{n-1} \ar[r]^{g_{n}f_{n}g_{n-1}^{-1}} & x_n} \]
for which the top horizontal arrows together with the vertical arrows uniquely determine the bottom arrows.\\

From this viewpoint is clear that we have an equivalence of groupoids
$$[[0,n],C]_{\overline{x}, \overline{x}_1, ...., \overline{x}_{n-1}, \overline{y}} \ \ \ \simeq \
 \ \  \left( \prod_{i=0}^{n-1} C(x_i, x_{i+1}) \right) \rtimes \left( \prod_{i=0}^{n} C(x_i, x_{i}) \right),$$
and thus we have that
$$\Big| [[0,n],C]_{\overline{x}, \overline{x}_1, ...., \overline{x}_{n-1}, \overline{y}} \Big|_{_{\mathrm{g}}} \ \ = \ \
\frac{|C(x_0, x_1)|.................|C(x_{n-1}, x_n)|}{|C(x_0,x_0)||C(x_1,x_1)|.......|C(x_{n-1},x_{n-1})||C(x_n,x_n)|} $$
and we have obtained the following result.

\begin{lem}\label{l3} {\em Let $C$ be an isocyclic essentially locally finite category and $\overline{x} < \overline{y}$ in $\overline{C}.$ We have that:
$$\mu[\overline{x},\overline{y}] \ = \
\sum_{n \geq 1}\sum_{\overline{x} =\overline{x}_0 < \overline{x}_1 < ... < \overline{x}_n=\overline{y}}
(-1)^n\Big| [[0,n],C]_{\overline{x}, \overline{x}_1, ...., \overline{x}_{n-1}, \overline{y}} \Big|_{_{\mathrm{g}}} \ = \
\sum_{n \geq 1}(-1)^n \Big|[[0,n],C]_{\overline{x},\overline{y}}^{\natural}\Big|_{_{\mathrm{g}}}  .$$
}
\end{lem}

\begin{proof} Follows from Theorem \ref{t3} and the formula above.
\end{proof}

For our next constructions we need a few notions from the theory of simplicial sets \cite{goe, m, w}.  We are going to show that the formula above for $\mu[\overline{x},\overline{y}]$ can be understood in terms of augmented simplicial essentially finite groupoids. A simplicial essentially finite groupoid is a  functor
$$G: \Delta^{\circ} \longrightarrow \mathrm{gpd},$$ where $\Delta^{\circ}$ is the opposite category of $\Delta. $ Objects in $\Delta$ are the intervals $[0,n] \subseteq \mathbb{N}$ for $n \in \mathbb{N}.$ Morphisms in $\Delta$ are
non-decreasing maps.\\

Thus a simplicial essentially finite groupoid $G$ assigns an essentially finite groupoid $G_n$ to each $n \in \mathbb{N}$, and a functor $$\widehat{f}:G_m \longrightarrow G_n $$
to each non-decreasing map $f:[0,n] \longrightarrow [0,m].$\\

An augmented simplicial essentially finite groupoid is a functor
$$G: \Delta_a^{\circ} \longrightarrow \mathrm{gpd}$$ where $\Delta_a$ is the category obtained from $\Delta$ by adjoining an object $[-1]$ and a
unique morphism $[-1] \longrightarrow [0,n]$ for each $n \geq -1.$ Thus an augmented simplicial groupoid $G$ has, in addition, a groupoid $G_{-1}$ and a unique functor $G_n \longrightarrow G_{-1}$ for each
$n \geq -1$.\\

For $G$ an augmented simplicial groupoid and $n \geq -1$, we let $G_n^{\natural}$ be the full subgroupoid of $G_n$ whose objects are the non-degenerated objects of $G_n$. An object $x \in G_n$ is called degenerated if there is a non-decreasing map $f: [0,n] \longrightarrow [0,m]$ with $m < n$, and an object
$y \in  G_m$ such that $\widehat{f}(y) \simeq x.$

\begin{defn}{\em The reduced Euler characteristic $\widetilde{\chi}_{\mathrm{g}}G$ of an augmented simplicial groupoid $G: \Delta_a^{\circ} \longrightarrow \mathrm{gpd}\ $ is given by
$$\widetilde{\chi}_{\mathrm{g}} G \ = \ \sum_{n \geq -1}(-1)^n|G_n^{\natural} |_{\mathrm{g}} .$$
}
\end{defn}

Let $C$ be an isocyclic essentially locally finite category. For  $\overline{x} < \overline{y}$ in $\overline{C},$ we define the augmented
simplicial groupoid $C_{\ast}(\overline{x}, \overline{y})$ as follows. For $n \geq -1$, we set
$$C_{n}(\overline{x},\overline{y}) \ = \ [[0,n+2],C]_{\overline{x},\overline{y}}. $$
Given a morphism $f:[0,n] \longrightarrow [0,m] \ $ in $\ \Delta$, consider the non-decreasing extension map $f_e$ defined trough the commutative diagram
\[\xymatrix @R=.4in  @C=1in
{[0,n] \ar[r]^{f} \ar[d] &  [0,m] \ar[d] \\
\{0\} \sqcup [1, n+1] \sqcup \{n+2\} \ar[r]^{f_e}  &  \{0\} \sqcup [1, m+1] \sqcup \{m+2\}  }\]
where the vertical arrows arise from the increasing bijections
$$[0,n] \ \longrightarrow \ [1,n+1]\ \ \ \ \ \mbox{and} \ \ \ \ \ [0,m]\ \longrightarrow \ [1,m+1],$$ and we set
$f_e(0)=0, \ \ f_e(n+2)=m+2. \ $
Using this notation the functor $$\widehat{f}: [[0,m+2],C]_{\overline{x},\overline{y}} \ \ \longrightarrow \ \  [[0,n+2],C]_{\overline{x},\overline{y}}$$ is given by
$$\widehat{f}(F)\ =\ F f_e.$$
If we have another  morphism $\ g:[0,m] \longrightarrow [0,k] \ $ in $ \Delta$, then we have that
$$(gf)_e=g_e f_e \ \ \ \ \ \mbox{and thus}$$
$$\widehat{g f}(F)\ =\ F(gf)_e \ = \ F(g_e f_e) \ = \ (F g_e)  f_e \ =\ (\widehat{f}  \widehat{g} )(F). $$
The required unique functor  $\ [[0,m+2],C]_{\overline{x},\overline{y}} \ \ \longrightarrow \ \ [[0,1],C]_{\overline{x},\overline{y}}\ \ $ sends a diagram
 \[\xymatrix @R=.3in  @C=.4in
{x_0 \ar[r]^{f_1} \ar[d] &  x_1 \ar[r]^{f_2\ \ \ \ \ } \ar[d] & \ \ \ \ \ \ \ .... \ \ \ \ \ar[r]^{\ \ f_{n+1}} & x_{n+1}  \ar[r]^{f_{n+2}} \ar[d] & x_{n+2} \ar[d]\\
z_0 \ar[r]^{g_1}  &  z_1 \ar[r]^{g_2\ \ \ \ \ }  &  \ \ \ \ \ \ \ .... \ \ \ \ \ar[r]^{\ \ g_{n+1}} &  z_{n+1} \ar[r]^{g_{n+2}} & z_{n+2}} \]
to the diagram
\[\xymatrix @R=.4in  @C=1.2in
{x_0 \ar[r]^{f_{n+2}f_{n+1}....f_{2}f_{1}} \ar[d] &  x_{n+2} \ar[d] \\
z_0 \ar[r]^{g_{n+2}g_{n+1}....g_{2}g_{1}}  & z_{n+2}  }\]
Thus  $C_{\ast}(\overline{x},\overline{y})$ is an augmented simplicial groupoid.
All together  we have shown the following result.

\begin{thm}\label{t77}{\em Let $C$ be an isocyclic essentially locally finite category and $\overline{x} < \overline{y}$ in $\overline{C}.$ Then
$$\mu[\overline{x},\overline{y}] \ = \ \widetilde{\chi}_{\mathrm{g}} C_{\ast}(\overline{x},\overline{y}).$$
}
\end{thm}

\begin{proof}
By Lemma \ref{l3} we have that
$$\mu[\overline{x},\overline{y}]\ = \
\sum_{n \geq 1}(-1)^n \Big| [[0,n],C]_{\overline{x},\overline{y}}^{\natural} \Big|_{\mathrm{g}} \ = $$
$$ \sum_{n \geq -1}(-1)^n \Big|[[0,n+2],C]_{\overline{x},\overline{y}}^{\natural}\Big|_{\mathrm{g}}  \ = \
\sum_{n \geq -1}(-1)^n \Big| C_{n}(\overline{x},\overline{y}) \Big|_{\mathrm{g}} \ = \ \widetilde{\chi}_{\mathrm{g}}C_{\ast}(\overline{x},\overline{y}).$$
\end{proof}

Next we consider  another distinguished element $\xi_{\mathrm{g}}$ in the incidence algebra $[\mathbb{I}_{(\overline{C},\leq)}, R]$ of an isocyclic essentially
locally finite category $C$ given for $\ \overline{x}\leq \overline{y} \in \overline{C} \ $  by:
$$\xi_{\mathrm{g}}[\overline{x},\overline{y}] \ = \ \Big|C(x,y)\rtimes (C(x,x)\times C(y,y)) \Big|_{\mathrm{g}}.$$ Thus we have that
$$\xi_{\mathrm{g}}[\overline{x},\overline{y}] \ = \frac{|C(x,y)|}{|C(x,x)||C(y,y)|} \ = \
 \frac{\xi[\overline{x},\overline{y}]}{|C(x,x)||C(y,y)|}.$$

For $\overline{x} < \overline{y} \in \overline{C}$, we let $D_{\ast}(x,y)$ be the augmented sub-simplicial groupoid of $C_{\ast}(\overline{x},\overline{y})$
such that for $n \geq -1$:
\begin{itemize}
\item  Objects in $D_{n}(x,y)$
are functors $F:[0,n+2]\longrightarrow C$ with $F(0)=x$ and $F(n+2)=y.$

\item  A morphism in $D_{n}(x,y)$ from $F$ to $G$ is a natural isomorphism $l: F \longrightarrow G$ such that the morphisms $l(0):x \longrightarrow x \ \ \mbox{and} \ \ l(n+2):y \longrightarrow y$ are identities.
\end{itemize}

Thus a morphism
in $D_{\ast}(x,y)$ is a commutative diagram of the form
\[\xymatrix @R=.2in  @C=.3in
{    &  x_1 \ar[r] \ar[dd] & x_2  \ar[r] \ar[dd] & \ \ \ \ \ \ \ .... \ \ \ \  \ar[r] &  x_{n} \ar[r]  \ar[dd] & x_{n+1} \ar[dd] \ar[dr] &  \\
x  \ar[ur] \ar[dr] &   & \ \ \ \ \ \ \  \ \ \ \  &  & & & y  \\
     &  z_1 \ar[r]  & z_2 \ar[r]  & \ \ \ \ \ \ \ .... \ \ \ \ \ar[r] & z_{n} \ar[r] &  z_{n+1} \ar[ur] &    } \]
where the diagonal and horizontal arrows are not allowed to be isomorphisms, and the vertical arrows are isomorphisms.

\begin{thm}\label{t7}{\em
The map $\xi_{\mathrm{g}}$ is a unit in $[\mathbb{I}_{(\overline{C},\leq)}, R]$  and its inverse $\mu_{\mathrm{g}}$ is given for
$\overline{x} < \overline{y}$ in $\overline{C}\ $ by
$$ \mu_{\mathrm{g}}[\overline{x},\overline{x}]\ = \ |C(x,x)|, \ \ \ \ \ \mbox{and}$$
$$\mu_{\mathrm{g}}[\overline{x},\overline{y}] \ = \ \sum_{n \geq 1}\ \sum_{\overline{x}  =
\overline{x}_0 < \overline{x}_1 < ... < \overline{x}_n=\overline{y}}
(-1)^n\frac{|C(x_0, x_1)|.................|C(x_{n-1}, x_n)|}{|C(x_1,x_1)|.......|C(x_{n-1},x_{n-1})|} .$$
Moreover, we have that:
$$\mu_{\mathrm{g}}[\overline{x},\overline{y}]  \ = \  \widetilde{\chi}_{\mathrm{g}} D_{\ast}(\overline{x},\overline{y}).$$}
\end{thm}

\begin{proof}It follows from Theorem \ref{l1} that:
$$\mu_{\mathrm{g}}[\overline{x},\overline{y}] \ = \ \sum_{n\geq 1}\sum_{\overline{x}  =
\overline{x}_0 < \overline{x}_1 < ... < \overline{x}_n=\overline{y}}
(-1)^n\frac{\xi_{\mathrm{g}}[x_0, x_1].................\xi_{\mathrm{g}}[x_{n-1}, x_n]}{\xi_{\mathrm{g}}[x_0,x_0]\xi_{\mathrm{g}}[x_1,x_1].......
\xi_{\mathrm{g}}[x_{n-1},x_{n-1}]\xi_{\mathrm{g}}[x_n,x_n]}\ =$$
$$\sum_{n\geq 1}\sum_{\overline{x}  =
\overline{x}_0 < \overline{x}_1 < ... < \overline{x}_n=\overline{y}}
(-1)^n\frac{\frac{|C(x_0,x_1)|}{|C(x_0,x_0)||C(x_1,x_1)|}.............
\frac{|C(x_{n-1},x_n)|}{|C(x_{n-1},x_{n-1})||C(x_n,x_n)|}}
{\frac{|C(x_0,x_0)|}{|C(x_0,x_0)||C(x_0,x_0)|}.................
\frac{|C(x_{n},x_{n})|}{|C(x_{n},x_{n})||C(x_{n},x_{n})|}}\ =$$
$$\sum_{n \geq 1}\ \sum_{\overline{x}  =
\overline{x}_0 < \overline{x}_1 < ... < \overline{x}_n=\overline{y}}
(-1)^n\frac{|C(x_0, x_1)|.................|C(x_{n-1}, x_n)|}{|C(x_1,x_1)|.......|C(x_{n-1},x_{n-1})|}.$$
\end{proof}

\section{M$\ddot{\mbox{o}}$bius Categories}

In this and the next section we adopt the viewpoint that regards $\mathbb{N}_+,$ in the classical M$\ddot{\mbox{o}}$bius theory, as a monoid. The main idea is to
develop a  M$\ddot{\mbox{o}}$bius theory that applies to a large class of monoids. It turns out that this goal can be readily achieved  for the category of finite decomposition monoids. Indeed one can go further and define a M$\ddot{\mbox{o}}$bius theory that applies to
finite decomposition categories, better known in the literature as  M$\ddot{\mbox{o}}$bius categories.
Since a monoid is just a category with one object, the theory of finite decomposition monoids embeds into the theory of M$\ddot{\mbox{o}}$bius categories.
The notion of  M$\ddot{\mbox{o}}$bius categories was introduced by Leroux in \cite{lero, lero2}, and since then there have been  quite a few publications
in the field, among them \cite{ha, law, le2, sop}. Here we limit ourselves to the most basic results on M$\ddot{\mbox{o}}$bius categories. \\

Let $C$ be a category and $f$ a morphism in $C$. A $n$-decomposition of $f$, for $n \geq 1$, is $n$-tuple
$(f_1,...,f_n)$  of morphisms in $C$ such that $$f_n.....f_1\ = \ f.$$
Thus a $n$-decomposition of a morphism $f:x \longrightarrow y$ is a commutative diagram of the form
\[\xymatrix @R=.3in  @C=.3in
{  &  x_1 \ar[r]^{f_2} & x_2  \ar[r]^{f_3\ \ \ \ \ \ } & \ \ \ \ \ \ \ .... \ \ \ \ \ \ \ \ \  \ar[r]^{\ \ f_{n-2}} &  x_{n-1} \ar[r]^{f_{n-1}}  & x_{n-1} \ar[dr]^{f_n} &  \\
x  \ar[ur]^{f_1} \ar[rrrrrr]^f &   & \ \ \ \ \ \ \ \ \ \   \ \ \ \  &  & &  & y  } \]
Let $\mathrm{D}_nf$ be the set of $n$-decompositions of $f$.
A decomposition is called proper if none of its components is an identity morphism.
For $n \geq 2$, let $\mathrm{PD}_nf$ be the set of proper $n$-decompositions of $f$. Set
$\ \mathrm{PD}_1f = \mathrm{D}_1f = \{ f\} \ $  and
$$\mathrm{D}f \ = \ \bigsqcup_{n \geq 1} \mathrm{D}_nf \ \ \ \ \ \mbox{and} \ \ \ \ \ \mathrm{PD}f \ = \ \bigsqcup_{n \geq 1} \mathrm{PD}_nf.$$

\begin{defn}\label{ddd}{\em A  category $C$  is M$\ddot{\mbox{o}}$bius if $\mathrm{PD}f$ is a finite set for all morphisms $f$ in $C$,
i.e. each morphism in $C$ admits a finite number of  proper decompositions.
}
\end{defn}

Note that in a M$\ddot{\mbox{o}}$bius category the only isomorphisms are the identities.

\begin{lem}{\em If $\mathrm{PD}f$ is a finite set, then $\mathrm{D}_nf$ is a finite set for all $n \geq 1$. Indeed we have:
$$\big| \mathrm{D}_nf\big| \ = \ \sum_{k \geq 1}^n{ n \choose k} \big| \mathrm{PD}_{k}f\big| \ \ \ \ \  \mbox{and}
\ \ \ \ \ \big|\mathrm{PD}_{n}f \big| \ = \ \sum_{k \geq 1}^n (-1)^{n-k} {n \choose k}\big| \mathrm{D}_kf \big|.$$
}
\end{lem}

\begin{proof}
The left identity is shown as follows. Out of the $n$ morphisms in a $n$-decomposition assume that $n-k$ are identities which can be placed in ${n \choose k} $ different positions.  Omitting the identity morphisms, a $n$-decomposition of the latter type reduces to a proper $k$-decomposition. The right identity follows from the M$\ddot{\mbox{o}}$bius inversion formula.
\end{proof}

Let $f$ be a morphism in a category $C$. We say that $f$ fixes a morphism $g \in C $ if
$$fg=g \ \ \ \ \ \mbox{or} \ \ \ \ \ gf =g.$$ Next result is due to Leroux \cite{lero, law}.

\begin{thm}\label{mej}
{\em A category $C$ is M$\ddot{\mbox{o}}$bius if and only if the following conditions hold:
\begin{itemize}
\item $\mathrm{PD}_2f$ is a finite set for each morphism $f$ in $C$.
\item Identities admit no proper decomposition.
\item If $f$ fixes a morphism, then $f$ is an identity morphism.
\end{itemize}
}
\end{thm}

\begin{proof}
Assume $C$ is M$\ddot{\mbox{o}}$bius category. By definition   $\mathrm{PD}_2f$ is a finite set for each $f\in C$. If an identity morphism $1$ in $C$ admits a proper decomposition $$f_n.....f_1 \ = \ 1,$$ then
it admits infinitely many proper decompositions, indeed
$$1\ = \ f_n.....f_1\ = \ f_n.....f_1f_n.....f_1\ = \ f_n.....f_1f_n.....f_1f_n.....f_1 \ = \ .....$$ Let $f$ be a non-identity morphism and $g$ be another morphism. If $fg=g$, then $g$ is a   non-identity morphism that admits infinitely many proper decompositions $$g \ = \ fg \ = \ ffg \ = \ fffg\ = \ fffg\ = \ .....$$ Thus we conclude that $fg \neq g,$ and a similar argument shows that $gf \neq g$ as well. \\

Suppose now that the three conditions of the theorem hold. An inductive argument shows that if $\mathrm{PD}_2f$ is a finite set, then $\mathrm{PD}_nf$ is a finite set for all $n \geq 2.$ Let
$k = |\mathrm{PD}_2f|$, we show that $f$ can not have a proper decomposition of length greater or equal $k+2$. Assume that we have such a decomposition $(f_1,...,f_{k+2})$. Composing initial and final segments of morphisms in $(f_1,...,f_{k+2})$ one obtains $k+1$ proper $2$-decompositions
of $f$. They can not be all different, thus there exist $g$ and $h$, obtained by composing some of the $f_i$, such that $gh=h$. Then $h$ must be an identity,
in contradiction with the fact that identities admit no proper decomposition.

\end{proof}

\begin{defn}\label{gd}
{\em
A $\mathbb{N}$-graded  category with finite graded components $C$ is a category such that for $x,y,z \in C$
we have $$C(x,y)=\bigsqcup_{n \in \mathbb{N}}C_n(x,y),$$
$$C_0(x,x)=\{ 1_x \}, \ \ \ \ \ |C_n(x,y)|< \infty \ \ \ \mbox{and} \ \ \ C_n(x,y) \times C_m(y,z) \longrightarrow C_{n+m}(x,z).$$ The degree $\mathrm{deg}f \in \mathbb{N}$ of a morphism $f:x \longrightarrow y$ is such that $f\in C_{\mathrm{deg}f}(x,y).$
}
\end{defn}

\begin{exmp}{\em Let $\Gamma$ be a finite directed graph.
Let $P_{\Gamma}$ be the category of paths in $\Gamma$, i.e. objects of $P_{\Gamma}$ are the vertices of $\Gamma$, and morphisms in $P_{\Gamma}(x,y)$ are paths in $\Gamma$ from $x$ to $y$. By convention there is an empty path of length $0$ from each vertex to itself. Let $R$ be an equivalence relation on $\Gamma$-paths generated by a collection of pairs of paths in $\Gamma$ of the same length and with the same endpoints. Let $P_R$ be the category whose objects are the vertices of $\Gamma$ and whose morphisms are equivalence classes of paths in $\Gamma$. The category $P_R$ is $\mathbb{N}$-graded  with finite graded components.}
\end{exmp}

\begin{prop}{\em
Let $C$ be a $\mathbb{N}$-graded  category with finite graded components, let $(X, \leq )$ be a locally finite poset, and $F: X \longrightarrow C$ be a functor.   Let $X_F$ be the category with $X$ as its set of objects and with morphisms given by
 $$X_F(x,y) \ = \
\left\{\begin{array}{cc} C(F(x),F(y)) & \ \mathrm{if} \  x \leq y, \\
\emptyset & \ \ \  \mathrm{otherwise.}  \end{array}\right.
$$
The category $X_F$ is  M$\ddot{\mbox{o}}$bius.}
\end{prop}

\begin{proof}Fix $x, y \in C.$ Note that $X_F(x,y) \neq \emptyset$ if and only if  $x \leq y$ as elements of $X$.  Since $(X, \leq)$ is a locally finite poset there is only a finite number of choices of objects
$x_1,....,x_{n-1} \in C$ for which there is a diagram $$x \overset{f_1}\longrightarrow x_1 \overset{f_2} \longrightarrow ....... \overset{f_{n-1}}\longrightarrow x_{n-1} \overset{f_n}\longrightarrow y \ \ \ \ \  \mbox{with}  \ \ \ \ \ f_n ..... f_1 \ = \ f .$$ Note that the latter identity implies that $$\mathrm{deg}f_1 \ + \ \mathrm{deg}f_2 \ + \ ..... \ + \ \mathrm{deg}f_n \ = \ \mathrm{deg}f,$$ so given the tuple $x_1,....,x_{n-1}$ there is only a finite number of diagrams as the above.
\end{proof}

\begin{defn}{\em
A category $C$ is one way if in any diagram $x \overset{f}\longrightarrow y \overset{g} \longrightarrow x $ in $C$, we have that $\ x=y\ $ and $\ f=g= 1_x.$}
\end{defn}

\begin{prop}{\em Let $C$ be a category with $C(x,y)$ a finite set for all $x,y \in C$. Then $C$ is a M$\ddot{\mbox{o}}$bius category if and only if $C$ is a locally finite and one way category.
}
\end{prop}

\begin{proof}
Assume $C$ is a locally finite category. Objects of $C$ are partially ordered according to Lemma \ref{order}. For $x \leq y$ in $C$ we let $C[x,y]$ be the set of linearly ordered subsets of the interval $[x,y]$. The map $\mathrm{PD}f \longrightarrow C[x,y]$ sending $(f_1,...,f_n)$ to the linearly ordered set
$$\{tf_1 < tf_2 < .... <  tf_{n-1}\} \ \subseteq \ [x,y]$$ has a finite codomain since $[x,y]$ if a finite set as $C$ is locally finite, it has finite fibers because morphism between objects of $C$ are finite sets, thus $\mathrm{PD}f$ is a finite set.\\

Assume now that $C$ is a M$\ddot{\mbox{o}}$bius category with finite sets of morphisms. The map
$$\bigsqcup_{f\in C(x,y)}\mathrm{PD}f \longrightarrow C[x,y]$$ is surjective and has a finite domain.
Thus the interval $[x,y]$ is a finite set. Next we show that any endomorphism $f:x \longrightarrow x$  in $C$ is an identity. Since $C(x,x)$ is a finite set, there are integers $n \geq 0$ and $k\geq 1$ such $$f^n=f^{n+k}=f^{n}f^{k}.$$ So $f^k$ fixes a morphism, and then $f^k=1$ by Theorem \ref{mej}.  Since $1$ admits no proper decomposition we have that $f=1.$
\end{proof}

Next we associate a convolution algebra to each  category $C$ such that $\mathrm{D}_2f$ is a finite set for all morphisms $f$ in $C$.

\begin{defn}{\em  Let $C$ be a category such that $\mathrm{D}_2f$ is a finite set for all morphism  $f \in C_1$. The convolution algebra of $C$ is the pair $([C_1, R], \star)$
where for $\alpha, \beta \in [C_1, R]$ the product $\alpha \star \beta : C_1 \longrightarrow R$ is given on $f \in C_1$ by
$$\alpha \star \beta(f) \ = \ \sum_{(f_1, f_2) \in \mathrm{D}_2f}\alpha(f_1)\beta(f_2) .$$}
\end{defn}

\begin{prop}{\em $([C_1, R], \star)$ is an associative algebra with unit the map $1 \in [C_1, R]$
such that $1(f)=1$ if $f$ is an identity morphism, and zero otherwise.}
\end{prop}

\begin{proof} Since $\mathrm{D}_2f$ is a finite set, then $\mathrm{D}_nf$ is a finite set for all $n \geq 2.$
The unit property is clear. Associativity follows from the identities
$$\alpha_1 \star (\alpha_2 \star \alpha_3)(f) \ = \
\sum_{(f_1,f_2,f_3) \in \mathrm{D}_3f} \alpha_1(f_1)\alpha_2(f_2)\alpha_3(f_3)
\ = \ (\alpha_1 \star \alpha_2) \star \alpha_3(f).$$
More generally we have that
$$\alpha_1 \star \alpha_2 \star ....\star \alpha_n(f) \ = \
\sum_{(f_1,...,f_n) \in \mathrm{D}_nf} \alpha_1(f_1)\alpha_2(f_2) ....\alpha_n(f_n).$$
\end{proof}

\begin{cor}{\em Let $C$ be a category with $\mathrm{D}_2f$ is a finite set for all morphism  $f \in C_1$. The free $R$-module $<C_1>$ generated by $C_1$ together with the $R$-linear maps
$$\Delta: <C_1> \longrightarrow <C_1> \otimes <C_1> \ \ \ \ \ \textrm{and} \ \ \ \ \ \epsilon: <C_1> \longrightarrow R$$ given on generators, respectively, by
$$\Delta f\ = \ \sum_{(f_1, f_2) \in \mathrm{D}_2f}f_1 \otimes f_2 \ \ \ \ \ \mbox{and} \ \ \ \ \
\epsilon f \ = \ \left\{\begin{array}{cc}
1 & \ \mathrm{if} \ f \ \mbox{is an identity},  \\

0 & \  \mathrm{if} \ x \neq y, \ \ \ \ \ \ \ \ \ \ \ \ \
\end{array}\right.$$
is a $R$-coalgebra.}
\end{cor}

\begin{thm}\label{mier}
{\em Let $C$ be a M$\ddot{\mbox{o}}$bius category. A map $\alpha \in [C_1, R]$ is a $\star$-unit if and only if $\alpha(1_x)$ is
a unit in $R$ for all $x \in C.$ Thus the map $\xi \in [C_1, R]$  constantly equal to $1$
is a unit in $([C_1, R], \star)$; its inverse $\mu$, called the  M$\ddot{\mbox{o}}$bius function of $C$, is such that
$\mu(1_x)=1$ for all $x \in C$, and if $f$ is a non-identity morphism in $C$ then we have that
$$\mu f\ =\ \sum_{n \geq 1 }(-1)^n \big| \mathrm{PD}_nf \big|.$$
}
\end{thm}

\begin{proof} Let $\alpha, \beta \in [C_1, R]$
be such that $\alpha \star \beta =1.$ Then for all $x \in C$ we have that
$$\alpha(1_x)\beta(1_x)\ = \ (\alpha \star \beta)(1_x) \ = \ 1, \ \ \ \
\mbox{thus} \ \ \alpha(1_x) \ \  \mbox{is a unit}.$$
Conversely, assume that  $\alpha(1_x)$ is a unit and write
$\alpha^{-1}(1_x) = \frac{1}{\alpha(1_x)} $. Let $f$ be a non-isomorphism in $C$,  we have
that
$$\alpha^{-1}f \ \ =\ \ \sum_{n\geq 1}\ \ \sum_{(f_1,....,f_n) \in \mathrm{PD}_nf}
(-1)^n\frac{\alpha(f_1).......\alpha(f_n)}{\alpha(1_{sf_1})\alpha(1_{tf_1})....\alpha(1_{tf_n})} ,$$
where $s,t: C_1 \longrightarrow C_0$ are, respectively, the source and target maps of $C$.
\end{proof}

\begin{cor}{\em Let $C$ be a M$\ddot{\mbox{o}}$bius category and let $R[[x_{f}]]$ be the $R$-algebra of formal power series in the variables $x_{f}$ with
$f$ a non-identity in $C_1$. The structural maps on $R[[x_{f}]]$ given, respectively, on generators by
$$\epsilon 1=1 \ \ \ \ \mbox{and}\ \ \ \ \epsilon x_{f}= 0,$$
$$\Delta 1=1\otimes 1 \ \ \ \ \mbox{and} \ \ \ \
\Delta x_{f} \ = \ 1\otimes x_{f} \ + \ \sum_{(f_1, f_2) \in \mathrm{D}_2f}x_{f_1} \otimes x_{f_2} \ + \
x_{f}\otimes 1,$$
$$S 1 =1 \ \ \ \ \mbox{and}\ \ \ \ S x_{f}\ = \ \sum_{n\geq 1}\ \sum_{(f_1,....,f_n) \in \mathrm{PD}_nf}(-1)^nx_{f_1}....x_{f_n}$$ turn $R[[x_{f}]]$ into a Hopf algebra, and the  M$\ddot{\mbox{o}}$bius $\mu \in [C_1, R]$ of $C$ is given by
$$\mu f \ = \ Sx_{f}(1).$$.}
\end{cor}

\begin{prop}{\em Let $C$ and $D$ be finite M$\ddot{\mbox{o}}$bius categories, then
$C\times D$ is a finite M$\ddot{\mbox{o}}$bius category, $(C\times D)_1 \simeq C_1 \times D_1$, and we have a natural isomorphism of algebras
$$([(C\times D)_1, R], \star)  \ \simeq \ ([C_1, R], \star) \otimes
 ([D_1, R], \star).$$
Moreover, we have that $\xi_{C \times D}(f,g) \ = \ \xi_C f  \xi_D g$ and thus
$$\mu_{C \times D}(f,g)\ = \ \mu_C f  \mu_D g.$$
}
\end{prop}

Given a non-identity morphism $f:x\longrightarrow y$  in $C$, we construct the augmented simplicial set $D_{\ast}f$, a variant of the classical bar resolution,  as follows. For $n \geq -1$ we have inclusions
$$D_{n}f\ \subseteq  \ [[0,n+2],C]_{x , y},$$
where a functor $\ F\in [[0,n+2],C]_{x , y}\ $ is in $ D_{n}f $ if and only if
$$F(0 \leq n+2) \ = \ f.$$
With the notation the decomposition of morphism of Definition \ref{ddd}, we have  that:
$$D_{n}f \ = \ \mathrm{D}_{n+2}f \ \ \ \ \ \ \mbox{and} \ \ \ \ \ \ D_{n}f^{\natural}\ = \ \mathrm{PD}_{n+2}f.$$

\begin{thm}{\em The  M$\ddot{\mbox{o}}$bius function $\mu \in [C_1, R]$ of a M$\ddot{\mbox{o}}$bius category $C$ is given on a non-isomorphism $f$ in $C$ by
$$\mu f \ = \ \widetilde{\chi}D_{\ast}f \ = \ \sum_{n \geq 0}(-1)^n\mathrm{rank}\widetilde{\mathrm{H}}_n|D_{\ast}f|. $$
}
\end{thm}

\begin{proof} By Theorem \ref{mier} that
$$\mu f \ = \ \sum_{n \geq 1 }(-1)^n \big| \mathrm{PD}_nf \big| \ = \ \sum_{n \geq 0}(-1)^n\mathrm{rank}\widetilde{\mathrm{H}}_n|D_{\ast}f|.$$
From this identity and the previous remarks we get that
$$\mu f \ = \ \sum_{n \geq 1 }(-1)^n \big| \mathrm{PD}_nf \big| \ = \ \sum_{n \geq -1 }(-1)^n \big| \mathrm{PD}_{n+2}f \big| \ = \
\sum_{n \geq -1 }(-1)^n \big| D_{n}f^{\natural} \big|  \ = \ \widetilde{\chi}D_{\ast}f.$$
\end{proof}

\begin{cor}{\em
The reduced homology groups $\widetilde{\mathrm{H}}_n|D_{\ast}f|$  of the space $|D_{\ast}f|$ are obtained from the differential complex over $\mathbb{Z}$ such that:
\begin{itemize}
\item It is generated in degree $n$ by the $(n+2)$-decompositions $[f_1,...,f_{n+2}]$ of $f$.
\item  The differential $d$ is given on generators by:
 $$d[f]=0 \ \ \  \mbox{in degree} -1,$$ and it is given in degree $n \geq 0$ by
$$d[f_1,...,f_{n+2}] \ = \ \sum_{k=1}^{n+1}(-1)^k[f_1,...,f_{k+1}f_k ,...f_{n+2}]    .$$
\end{itemize}
}
\end{cor}

We close this section with a brief discussion of the relation between the M$\ddot{\mbox{o}}$bius theory for posets and the M$\ddot{\mbox{o}}$bius theory for finite decomposition categories.

\begin{prop}{\em
Let $C$ be a M$\ddot{\mbox{o}}$bius category.  The relation $\leq$ on $C_1$ given by
$$f\leq g \ \  \mbox{if and only if there is morphism}\ h\in C_1 \  \mbox{such that}\ g=hf,$$
is a partial order on $C_1$.
}
\end{prop}

\begin{proof}
Reflexivity and transitivity are valid for arbitrary categories. If $f \leq g \leq f,$
then there are morphisms $h$ and $k$ in $C_1$ such that $g=hf$ and $f=kg$. Thus $f=(kh)f$, and since $C$
is M$\ddot{\mbox{o}}$bius and $kh$ fixes $f$, then  $kh$ is an identity $1$ in $C$. Moreover, since identities in $C$ can not be properly decomposed, then $k=h=1$  by Theorem \ref{mej}, and thus $f=g$.
\end{proof}

\begin{defn}{\em
A category $C$ is right cancellative if for any morphisms $f,g,h$ in $C$ we have:
$$gf=hf \ \ \ \ \mbox{implies that}  \ \ \ \ g=h.$$
}
\end{defn}
\begin{thm}{\em
Let $C$ be a cancellative  M$\ddot{\mbox{o}}$bius   category. The convolution algebra
$[C_1, R]$ is isomorphic to the subalgebra  $$[\mathbb{I}_{(C_1, \leq )}, R]_c \ \subseteq \ [\mathbb{I}_{(C_1, \leq )}, R]$$ consisting of maps
$$\alpha : \mathbb{I}_{(C_1, \leq )} \longrightarrow R \ \ \  \mbox{such that} \ \ \  \alpha[f, gf]=\alpha[1,g]  \ \ \ \mbox{for all morphisms} \ \ \ f,g \in C_1.$$

}
\end{thm}

\begin{proof}Consider the map $[C_1, R] \longrightarrow [\mathbb{I}_{(C_1, \leq )}, R]$ sending a map
$\beta: C_1 \longrightarrow  R $ to the map $\widehat{\beta}: \mathbb{I}_{(C_1, \leq )} \longrightarrow  R$ given by $$\widehat{\beta}[f,hf] \ = \ \beta(h).$$
The map $\widehat{\beta}$ is well-defined since $C$ is cancellative and thus $f \leq g$ if and only if there exists an unique morphism $h=\frac{f}{g}$ such that $g=hf.$ It is clear that the image of $\widehat{\beta}$ is included in $[\mathbb{I}_{(C_1, \leq )}, R]_c$, we show that is actually equal to
$[\mathbb{I}_{(C_1, \leq )}, R]_c$. For $\alpha \in [\mathbb{I}_{(C_1, \leq )}, R]_c$, let $\beta \in [C_1, R]$ be given by $\beta(f)=\alpha[1,f]$. Then we have that
$$\widehat{\beta}[f,hf] \ = \ \beta(h) \ = \ \alpha[1,h] \ = \ \alpha[f,hf], \ \ \ \ \ \mbox{that is}
\ \ \ \ \ \widehat{\beta}=\alpha.$$ On the other hand assume that $\widehat{\beta_1} = \widehat{\beta_2}$, then $$\beta_1(h)\ = \ \widehat{\beta_1}[1,h] \ = \  \widehat{\beta_2}[1,h] \ = \ \beta_2(h).$$
So the map $\beta \longrightarrow \widehat{\beta}$ is injective. It remains to show that it is also an algebra map.
$$\widehat{\beta_1}\widehat{\beta_2}[1,g] \ = \  \sum_{1 \leq f \leq g}\widehat{\beta_1}[1,f]\widehat{\beta_2}[f,g]
\ = \  \sum_{1 \leq f \leq g}\beta_1(f)\beta_2(\frac{f}{g}) \ = $$
$$ \sum_{hf=g}\beta_1(f)\beta_2(h)\ = \  \beta_1\beta_2(g) \ = \ \widehat{\beta_1\beta_2}[1,g] .$$
\end{proof}

\section{Essentially Finite Decomposition Categories}

The notion of M$\ddot{\mbox{o}}$bius categories discussed in the previous section is invariant under isomorphisms of categories, but fails to be invariant under equivalence of categories. Indeed, as we have already remarked M$\ddot{\mbox{o}}$bius categories are skeletal. In this section we study a variant notion for which this issue  is circumvented.\\

Let $C$ be a category and $\overline{C}_1$ be the set of isomorphisms classes of morphisms in $C$. Recall that morphisms $f,g \in C_1$ are
isomorphic if they fit into a commutative diagram
\[\xymatrix @R=.35in  @C=.65in
{x \ar[r]^{f} \ar[d] &  y \ar[d] \\
z \ar[r]^{g}  & w  }\]
where the vertical arrows are isomorphisms.\\

\begin{defn}\label{bbbb}{\em
Fix $\overline{f} \in \overline{C}_1$. A $n$-decomposition of $\overline{f}$, for $n \geq 1$, is a $n$-tuple
$(f_1,...,f_n)$  of morphisms in $C$ such that there are isomorphisms $\alpha$ and $\beta$ in $C$ for which  $$f_n.....f_1\ = \ \beta^{-1} f\alpha,$$
that is a $n$-decompostion of $\overline{f}$ is given by a commutative diagram of the form
 \[\xymatrix @R=.4in  @C=.8in
{x_0 \ar[r]^{f_1} \ar[d]_{\alpha} &  x_1 \ar[r]^{f_2\ \ \ \ \ } & \ \ \ \ \ \ \ .... \ \ \ \ \ar[r]^{\ \ f_{n-1}} & x_{n-1}  \ar[r]^{f_n} & x_n \ar[d]_{\beta} \\
x \ar[rrrr]^{f}  &    &   &   & y} \]}
\end{defn} with $\alpha$ and $\beta$ isomorphisms.\\

Let $\mathrm{D}_n\overline{f} \ $ be the groupoid whose objects are $n$-decompositions of $\overline{f}$, and whose morphisms are commutative diagrams of the form
\[\xymatrix @R=.3in  @C=.6in
{ x_0 \ar[r]^{f_1} \ar[d]  &  x_1 \ar[r]^{f_2 \ \ \ \ \ \ \ } \ar[d] & \ \ \ \ \ \ \ .... \ \ \ \   \ar[r]^{\ \ \ \ f_{n-1}}  & x_{n-1} \ar[d] \ar[r]^{f_n} & x_n \ar[d] \\
z_0 \ar[r]^{\ g_1}      &  z_1 \ar[r]^{\ g_2 \ \ \ \ \ \ \ }  & \ \ \ \ \ \ \ .... \ \ \ \  \ar[r]^{\ \ \ \ \ g_{n-1}} &  z_{n-1} \ar[r]^{\ g_n} &   z_n  } \]
where the top and bottom of the diagram are $n$-decompositions of $\overline{f}$ and the vertical arrows are isomorphisms.\\

\begin{defn}\label{cccc}{\em
For $n \geq 2$, a decomposition $(f_1,...,f_n)$ of $\overline{f}$ is called proper if none of the morphisms $f_i$ is an isomorphism.
For $n \geq 2$, we let $\mathrm{PD}_n\overline{f}$ be the full subgroupoid of $\mathrm{D}_n\overline{f}$ whose objects are proper decompositions.
We set $\mathrm{PD}_1\overline{f} = \mathrm{D}_1\overline{f} \ $  and
$$\mathrm{D}\overline{f} \ = \ \bigsqcup_{n \geq 1} \mathrm{D}_n\overline{f}\ \ \ \ \ \mbox{and} \ \ \ \ \ \mathrm{PD}\overline{f} \ = \ \bigsqcup_{n \geq 1} \mathrm{PD}_n\overline{f},$$
i.e. $(\mathrm{PD}\overline{f}) \ \ \mathrm{D}\overline{f} $  is the groupoid of all (proper) decompositions of $\overline{f}$.}
\end{defn}

\begin{defn}{\em An essentially finite decomposition category $C$ is a category such that $\overline{\mathrm{PD}} \overline{f}$ is a finite set for  $\overline{f} \in \overline{C}_1$.
}
\end{defn}

\begin{lem}{\em If $\overline{\mathrm{PD}}\overline{f}$ is a finite set, then $\overline{\mathrm{D}}_n\overline{f}$ is a finite set for  $n \geq 1$ and we have that:
$$\big| \overline{\mathrm{D}}_n\overline{f}\big| \ = \ \sum_{k \geq 1}^n{ n \choose k} \big| \overline{\mathrm{PD}}_{k}\overline{f}\big| \ \ \ \ \ \ \mbox{and}
\ \ \ \ \ \ \big|\overline{\mathrm{PD}}_{n}\overline{f} \big| \ = \ \sum_{k \geq 1}^n (-1)^{n-k}{n \choose k} \big| \overline{\mathrm{D}}_k\overline{f} \big|.$$
}
\end{lem}

\begin{proof}
The identity on the left is shown as follows. Out of the $n$ morphisms in a $n$-decomposition assume that $n-k$ are isomorphism, they can be placed in ${n \choose k} $ different positions. Such a $n$-decomposition is isomorphic to a $n$-decomposition
where the isomorphisms are replaced by identities and the remaining morphisms are not isomorphisms. Omitting the identity arrows
a $n$-decomposition of the latter type reduces to a proper $k$-decomposition. The identity on the right follows from M$\ddot{\mbox{o}}$bius inversion
formula.
\end{proof}

\begin{thm}{\em Assume $C$ is a category with $C(x,y)$ finite for $x,y \in C$. Then $C$ is an essentially finite decomposition category if and only if $C$ is an
isocyclic essentially locally finite category.
}
\end{thm}

\begin{proof}Let $C$ be an essentially finite decomposition category.
Let $f: x \longrightarrow  x$ be an endomorphism in $C$, we show that $f$ is an isomorphism. Since $C(x,x)$ is a finite set and
$$\{1, f, f^2, .... ,f^n,... \} \ \subseteq \ C(x,x),$$
there must be $n \geq 0$ and $k \geq 1$ such that $f^n=f^{n+k}$.
If $n=0$, then  $f^k=1$ for some $k \geq 1$ and thus $f$ is an isomorphism. If $n>0$ and  $f$ is not an isomorphism then the identities
$$f^n=f^{n+k}=f^{n +2k}=f^{n+3k}=.....$$
show that $f^n$ admits infinitely many decompositions,  a contradiction. Thus in any case $f$ has to be isomorphism.\\

Next we show that the intervals $[\overline{x},\overline{y}]$ in $\overline{C}$, with the order coming from Lemma \ref{qorder}, are finite sets. The map
$$\{\overline{x},\overline{y}\} \ \sqcup \ \bigsqcup_{f \in C(x,y)}\overline{\mathrm{PD}}_2\overline{f}  \ \longrightarrow \ [\overline{x},\overline{y}],$$
being the identity on $\{\overline{x},\overline{y}\}$ and sending a $2$-decomposition $x \longrightarrow z \longrightarrow y$ of $\overline{f}$ to $\overline{z}$, is surjective and
has a finite domain because $C$ is a finite decomposition category and $C(x,y)$ is a finite set. Thus  $[\overline{x},\overline{y}]$ is a finite set and
$C$ is an isocyclic essentially locally finite category.\\

Assume now that $C$ is an isocyclic essentially locally finite category. Isomorphisms in $C$ admit a unique decomposition modulo equivalences. Given a non-isomorphism $f \in C(x,y)$ and
a chain $$\overline{x} < \overline{x}_1 < ... <\overline{x}_{n-1} < \overline{y} \ \ \mbox{in}\ \ \overline{C}, \ \ \ \mbox{we let} \ \ \
\mathrm{PD}(\overline{f},\overline{x}_1 , ... ,\overline{x}_{n-1}  ) $$
be the full subcategory of $\mathrm{PD}\overline{f}$ whose objects are proper $n$-decompositions of $f$ with the specified isomorphism classes for the intermediate objects.  We have that
$$\mathrm{PD}\overline{f} \ = \
\bigsqcup_{n\geq 2}\ \bigsqcup_{\overline{x} < \overline{x}_1 < ... <\overline{x}_{n-1} < \overline{y}}
\mathrm{PD}(\overline{f},\overline{x}_1 , ... ,\overline{x}_{n-1}).$$
Since  $[\overline{x},\overline{y}]$ is a finite set, it has a finite number of increasing chains. Since there are a finite number of morphisms
between objects of $C$, the sets $\overline{\mathrm{PD}}(\overline{f},\overline{x}_1 , ... ,\overline{x}_{n-1})$ are finite. Therefore, $\overline{\mathrm{PD}}\overline{f}$
is a finite set, and $C$ an essentially finite decomposition category.
\end{proof}

Next we  associate a convolution algebra for a category $C$ such that $\overline{\mathrm{D}}_2\overline{f}$ is a finite set for $\overline{f} \in \overline{C}_1$. The convolution product $\star$ on $[\overline{C}_1, R]$ is given on
$\alpha, \beta \in [\overline{C}_1, R]$  by
$$\alpha \star \beta(\overline{f}) \ = \ \sum_{\overline{(f_1, f_2)} \in \overline{\mathrm{D}}_2\overline{f}}\alpha(\overline{f}_1)\beta(\overline{f}_2) .$$

Associativity for the product $\star$  is by no means obvious. Thus we impose a (fairly strong) condition on $C$ that guarantees associativity.
In the midst of the proof of Theorem \ref{igu} below we provided a weaker (though less intuitive) condition that also guarantees associativity.

\begin{defn}{\em We say that a category $C$ has the isomorphism filling property if any commutative diagram
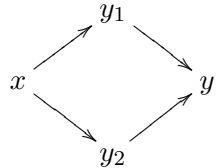
\[\xymatrix @R=.2in  @C=.3in
{  & y_1 \ar[dr] & \\
 x \ar[ur] \ar[dr]   &  & y\\
   &  y_2 \ar[ur]   &      } \]
with $y_1$ and $y_2$ isomorphic objects in $C$, can be enhanced to a commutative diagram
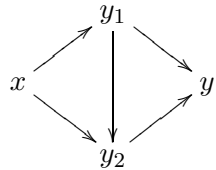
\[\xymatrix @R=.2in  @C=.3in
{  & y_1 \ar[dr]  \ar[dd]& \\
 x \ar[ur] \ar[dr]   &  & y\\
   &  y_2 \ar[ur]   &      } \]
where the vertical arrow is an isomorphism.}
\end{defn}

\begin{exmp}{\em The category $\mathbb{I}$ of finite sets and injective maps is isomorphism filling.
}
\end{exmp}

\begin{exmp}{\em The category $\mathrm{vect}$ of finite dimensional vector spaces and injective linear maps is isomorphism filling.
}
\end{exmp}

These previous examples follow from a general construction which we proceed to describe.
Let $\mathbb{B}$ be the category of finite sets and bijections, $F: \mathbb{B} \longrightarrow \mathrm{set}$ a functor (a.k.a. a combinatorial species \cite{aguiar, cd, re2, j1}). Let $\mathrm{Par}_F : \mathbb{B} \longrightarrow \mathrm{set}$ be the species of $F$-colored partitions, i.e.  for $x \in \mathbb{B} $ we let $\mathrm{Par}_Fx$ be the set of pairs $(\pi, s)$ where $\pi$ is a partition of $x$ and $s$ assigns  to each $b \in \pi$ an element $s_b \in Fb.$ Consider the category  $\mathrm{EPar}_F$ whose objects are triples $(x,\pi,s)$ with $(\pi, s) \in \mathrm{Par}_F x.$ A morphism $$f :(x,\pi,s) \longrightarrow (z,\sigma,u) \ \ \ \ \mbox{in} \ \ \ \mathrm{EPar}_F$$
is an injective map $f: x \longrightarrow z$ such that there is a triple $(y, \rho, t) \in \mathrm{EPar}_F$ such that
$$(fx \sqcup y,f\pi \sqcup \rho ,fs \sqcup t)\ = \ (z,\sigma,u).$$

\begin{lem}{\em The category $\mathrm{EPar}_F$ is isomorphism filling.}
\end{lem}
\begin{proof} Objects of $\mathrm{EPar}_F$ are triples $(x,\pi,s)$ which we denote just by $s$, since the map $s$ already includes the information about $x$ and $\pi$. The category $\mathrm{EPar}_F$ is monoidal with disjoint union as product, and it is complemented in the sense that for any morphism
$f :s \longrightarrow u$  there is a complement $u\setminus_{f} s \in \mathrm{EPar}_F$ such that
$$s \ \sqcup \ u\setminus_{f} s \ \simeq \ u.$$ Moreover, if we have morphisms $s \longrightarrow t_1 \ $ and $\ x \longrightarrow t_2$, with $t_1$ and $t_2$ isomorphic, then it follows that $t_1 \setminus s$ and $
t_2\setminus x$ are isomorphic objects and thus there is an isomorphism from $t_1$ to $t_2$ respecting the given morphisms.
\end{proof}

\begin{thm}\label{igu}{\em Let $C$ be an  isomorphism filling category with $\overline{\mathrm{D}}_2\overline{f}$ a finite set
for $\overline{f} \in \overline{C}_1$. The product $\star$ turns
$[\overline{C}_1, R]$ into an associative algebra with unit  $1 \in [\overline{C}_1, R]$ the map
such that $1(\overline{f})=1\ $ if $\ f$ is an identity morphism, and zero otherwise.}
\end{thm}

\begin{proof} The unit property is clear. Associativity follows from the identities
$$  (\alpha_1 \star \alpha_2) \star \alpha_3(\overline{f}) \ = \
\sum_{\overline{(f_1,f_2,f_3)} \in \overline{\mathrm{D}}_3\overline{f}} \alpha_1(\overline{f}_1)\alpha_2(\overline{f}_2)\alpha_3(\overline{f}_3)
\ = \ \alpha_1 \star (\alpha_2 \star \alpha_3)(\overline{f}) .$$
We show the left-hand side identity, the right-hand side identity is proven similarly. By definition we have that $$ (\alpha_1 \star \alpha_2) \star \alpha_3(\overline{f}) \ = \ \sum_{ \overline{(f_1,f_2)} \in \overline{\mathrm{D}}_2\overline{f} } (\alpha_1\star \alpha_2)(\overline{f}_1)\alpha_3(\overline{f}_2)\ = $$
$$\sum_{ \overline{(f_1,f_2)} \in \overline{\mathrm{D}}_2\overline{f} }\ \
\sum_{ \overline{(f_3,f_4)} \in \overline{\mathrm{D}}_2\overline{f}_1 } \alpha_1(\overline{f}_3)\alpha_1(\overline{f}_4)\alpha_3(\overline{f}_2). $$
Setting $$E \ = \ \left\{ (\overline{(f_1,f_2)},
 \overline{(f_3,f_4)} ) \ \ | \ \ \overline{(f_1,f_2)} \in \overline{\mathrm{D}}_2\overline{f} \  \ \mbox{and}\ \ \overline{(f_3,f_4)} \in \overline{\mathrm{D}}_2\overline{f}_1 \right\},$$
the desired result follows if there is a bijection
$\tau: \overline{\mathrm{D}}_3\overline{f}  \longrightarrow E$
making the diagram
\[\xymatrix @R=.5in  @C=.001in
{ \overline{\mathrm{D}}_3\overline{f}  \ar[dr] \ar[rr]^{\tau} & &  E \ar[dl] \\
& <\overline{C}_1> \otimes <\overline{C}_1> \otimes  <\overline{C}_1> &   }\]
commutative, where the diagonal arrows are given, respectively, by
$$\overline{(f_1,f_2,f_3)} \longrightarrow   \overline{f}_1 \otimes \overline{f}_2 \otimes \overline{f}_3, \ \ \ \ \ \mbox{and} $$
$$(\overline{(f_1,f_2)}, \overline{(f_3,f_4)} ) \longrightarrow  \overline{f}_3 \otimes \overline{f}_4 \otimes \overline{f}_2.$$
The desired map $\tau: \overline{\mathrm{D}}_3\overline{f}  \longrightarrow E$ is given by
$$\overline{(f_1,f_2,f_3)} \longrightarrow (\overline{(f_2f_1,f_3)}, \overline{(f_1,f_2)} ).$$
Let us show that $\tau$ is well-defined, surjective, and injective.\\

Let $\alpha_0, \alpha_1, \alpha_2, \alpha_3$ be isomorphisms from the appropriate objects so that
$$\overline{(f_1,f_2,f_3)} \ = \ \overline{(\alpha_1f_1\alpha_0^{-1},\alpha_2f_2\alpha_1^{-1},\alpha_3f_3\alpha_2^{-1})} .$$ The following
identities show that $\tau$ is well-defined:
$$\tau  \overline{(\alpha_1f_1\alpha_0^{-1},\alpha_2f_2\alpha_1^{-1},\alpha_3f_3\alpha_2^{-1})}  \ = \
(\overline{(\alpha_2f_2f_1\alpha_0^{-1},\alpha_3f_3\alpha_2^{-1})}, \overline{(\alpha_1f_1\alpha_0^{-1},\alpha_2f_2\alpha_1^{-1})} ) \ = $$
$$(\overline{(f_2f_1,f_3)}, \overline{(f_1,f_2)} ) \ = \ \tau \overline{(f_1,f_2,f_3)} .$$
To show that $\tau$ is surjective take a tuple $(\overline{(f_1,f_2)}, \overline{(f_3,f_4)} ) $ in $E$.  We may assume that $$f_2f_1 \ = \ f \ \ \ \ \ \ \mbox{and} \ \ \ \ \ \   f_4f_3\ =\ f_1.$$ So we have that
$$\tau\overline{(f_1,f_2,f_3)} \ = \ (\overline{(f_1,f_2)}, \overline{(f_3,f_4)} ).$$
To show injectivity we proceed as follows. Suppose we are given morphisms
$$x_0 \overset{f_1}{\longrightarrow} x_1 \overset{f_2}{\longrightarrow} x_2 \overset{f_{3}}{\longrightarrow} x_{3} \ \ \ \ \ \
\mbox{and}\ \ \ \ \ \ y_0 \overset{g_1}{\longrightarrow} y_1 \overset{g_2}{\longrightarrow} y_2 \overset{g_{3}}{\longrightarrow} y_{3},$$
such that
$$\tau\overline{(f_1,f_2,f_3)}  \ = \ \tau\overline{(g_1,g_2,g_3)}  \  \in  \overline{\mathrm{D}}_3\overline{f}. $$ Then
$$ (\overline{(f_2f_1,f_3)}, \overline{(f_1,f_2)} )\ = \ (\overline{(f_2f_1,f_3)}, \overline{(f_1,f_2)} ),$$ and thus  we have commutative diagrams with vertical isomorphisms:
 \[\xymatrix @R=.4in  @C=.8in
{x_0 \ar[r]^{f_1} \ar[d]_{\alpha_0} &  x_1 \ar[r]^{f_2}   & x_{2}  \ar[d]_{\alpha_2} \ar[r]^{f_2}   & x_{3}  \ar[d]_{\alpha_3} \\
y_0 \ar[r]^{g_1}  &  y_1 \ar[r]^{g_2}   & y_{2} \ar[r]^{g_2} & y_{3}  } \]
\[\xymatrix @R=.4in  @C=.8in
{x_0 \ar[r]^{f_1} \ar[d] &  x_1 \ar[r]^{f_2}\ar[d]   & x_{2}  \ar[d]    \\
y_0 \ar[r]^{g_1}  &  y_1 \ar[r]^{g_2}   & y_{2}   } \]
In particular  $x_1 \simeq y_1$ and thus we obtain the isomorphism filling $\alpha_1$ in the diagram
\[\xymatrix @R=.3in  @C=.4in
{  & x_1 \ar[dr]^{f_2}  \ar[dd]_{\alpha_1}& \\
 x_0 \ar[ur]^{f_1} \ar[dr]_{g_1\alpha_0}   &  & x_2\\
   &  y_1 \ar[ur]_{\alpha_2^{-1} g_2}   &      } \]
So we get the diagram
 \[\xymatrix @R=.4in  @C=.8in
{x_0 \ar[r]^{f_1} \ar[d]_{\alpha_0} &  x_1 \ar[r]^{f_2} \ar[d]_{\alpha_1}  & x_{2}  \ar[d]_{\alpha_2} \ar[r]^{f_2}   & x_{3}  \ar[d]_{\alpha_3} \\
y_0 \ar[r]^{g_1}  &  y_1 \ar[r]^{g_2}   & y_{2} \ar[r]^{g_2} & y_{3}  } \]
showing injectivity.

\end{proof}

More generally, we have the following result by induction.

\begin{cor}{\em
$$\alpha_1 \star \alpha_2 \star ....\star \alpha_n(\overline{f}) \ = \
\sum_{\overline{(f_1,...,f_n)} \in \overline{\mathrm{D}}_n \overline{f}} \alpha_1(\overline{f}_1)\alpha_2(\overline{f}_2) ....\alpha_n(\overline{f}_n).$$
}
\end{cor}

\begin{cor}{\em Let $C$ be an  isomorphism filling category with $\overline{\mathrm{D}}_2\overline{f}$ a finite set for $\overline{f} \in \overline{C}_1$. The free $R$-module $<\overline{C}_1>$ generated by $\overline{C}_1$ together with the $R$-linear maps
$$\Delta: <\overline{C}_1>\ \longrightarrow \ <\overline{C}_1> \otimes <\overline{C}_1> \ \ \ \ \ \textrm{and} \ \ \ \ \ \epsilon: <\overline{C}_1> \ \longrightarrow \ R$$ given on generators, respectively, by
$$\Delta \overline{f}\ = \ \sum_{(\overline{f}_1, \overline{f}_2) \in \overline{\mathrm{D}}_2\overline{f}}\overline{f}_1 \otimes \overline{f}_2 \ \ \ \ \ \mbox{and} \ \ \ \ \
\epsilon \overline{f} \ = \ \left\{\begin{array}{cc}
1 & \ \mathrm{if} \ f \ \mbox{is an isomorphism},  \\

0 & \  \mathrm{if} \ \mbox{otherwise}, \ \ \ \ \ \ \ \ \ \ \ \ \ \
\end{array}\right.$$
is a $R$-coalgebra.}
\end{cor}

\begin{thm}\label{yc}{\em Let $C$ be an essentially finite decomposition isomorphism filling category. A map $\alpha \in [\overline{C}_1,R]$ is a $\star$-unit if and only if $\alpha(\overline{1}_x)$ is
a unit in $R$ for all $x \in C.$ Thus the map $\xi \in [\overline{C}_1, R]$  constantly equal to $1$
is a unit in $([\overline{C}_1, R], \star)$; its inverse $\mu$ is called the  M$\ddot{\mbox{o}}$bius function of $C$ and is such that
$\mu\overline{f}=1$ if $f$ is an isomorphism, and if $f$ is not an isomorphism then $\mu\overline{f}$ is given by
$$\mu\overline{f} \ = \ \sum_{n \geq 1 }(-1)^n \big| \overline{\mathrm{PD}}_n\overline{f} \big|.$$
}
\end{thm}

\begin{proof} Assume $\alpha \star \beta =1.$ Then for all $x \in C$ we have that
$$\alpha(\overline{1}_x)\beta(\overline{1}_x)\ = \ (\alpha \star \beta)(\overline{1}_x)\ = \ 1, \ \ \ \
\mbox{thus} \ \ \ \ \alpha(\overline{1}_x)\ \ \ \mbox{is a unit}.$$
Conversely, if $\alpha(\overline{1}_x)$ is a unit  for all $x\in C$, then $\alpha$ is invertible and its inverse $\alpha^{-1}$ is given by
$$\alpha^{-1}(\overline{1}_x) = \frac{1}{\alpha(\overline{1}_x)} ,$$ and if $f$ a non-isomorphism then we have
that
$$\alpha^{-1}(\overline{f}) \ = \ \sum_{n\geq 1}\ \ \sum_{ \overline{(f_1,....,f_n)} \in \overline{\mathrm{PD}}_n\overline{f}}(-1)^n
\frac{\alpha(\overline{f}_1).......\alpha(\overline{f}_n)}{\alpha(\overline{1}_{sf_1})\alpha(\overline{1}_{tf_1})....\alpha(\overline{1}_{tf_n})} .$$
\end{proof}

\begin{cor}{\em Let $C$ be an essentially finite decomposition isomorphism filling category and let $R[[x_{\overline{f}}]]$ be the $R$-algebra of formal power series in the variables $x_{\overline{f}}$ for $\overline{f} \in \overline{C}_1$, with $f$ a non-isomorphism in $C_1$. The structural maps on $R[[x_{\overline{f}}]]$ given, respectively, on generators by
$$\epsilon 1=1 \ \ \ \ \mbox{and}\ \ \ \ \epsilon x_{\overline{f}}= 0,$$
$$\Delta 1=1\otimes 1 \ \ \ \ \mbox{and} \ \ \ \
\Delta x_{f} \ = \ 1\otimes x_{\overline{f}} \ + \ \sum_{\overline{(f_1, f_2)} \in \overline{\mathrm{PD}}_2\overline{f}}x_{\overline{f}_1} \otimes x_{\overline{f}_2} \ + \
x_{\overline{f}}\otimes 1,$$
$$S 1 =1 \ \ \ \ \mbox{and}\ \ \ \ S x_{\overline{f}}\ = \ \sum_{n\geq 1}\ \sum_{\overline{(f_1,....,f_n)} \in \overline{\mathrm{PD}}_nf}(-1)^nx_{\overline{f}_1}....x_{\overline{f}_n}$$ turn $R[[x_{\overline{f}}]]$ into a Hopf algebra,  and the M$\ddot{\mbox{o}}$bius function $\mu \in [C_1,R] $ of $ C$ is given by
$$\mu \overline{f} \ = \ Sx_{\overline{f}}(1).$$.}
\end{cor}

\begin{prop}{\em Let $C$ and $D$ be finite essentially finite decomposition isomorphism filling categories, then
$C\times D$ is a finite an essentially finite decomposition isomorphism filling category, $$\overline{(C\times D)}_1 \simeq \overline{C}_1 \times \overline{D}_1,$$ and we have a natural isomorphism of algebras
$$([(\overline{C\times D})_1, R], \star)  \ \simeq \ ([\overline{C}_1, R], \star) \otimes
 ([\overline{D}_1, R], \star).$$
Moreover, we have that $$\xi_{C \times D}\overline{(f,g)} \ = \ \xi_C \overline{f}  \xi_D \overline{g}$$ and thus
$$\mu_{C \times D}\overline{(f,g)}\ = \ \mu_C \overline{f}  \mu_D \overline{g}.$$
}
\end{prop}

Let $f:x\longrightarrow y$ be  a non-isomorphism in $C_1$.  We construct the augmented simplicial essentially finite groupoid $D_{\ast}\overline{f}$ which is
a full sub-simplicial groupoid of $C_{\ast}(\overline{x}, \overline{y})$. For $n \geq -1$, we have inclusions
$$D_{n}\overline{f}\ \subseteq \ C_{n}(\overline{x}, \overline{y})\ = \ [[0,n+2],C]_{\overline{x},\overline{y}}.$$

A functor $F\in C_{n}(\overline{x},\overline{y})$ belongs to $D_{n}\overline{f}\ $ if and only if
$$ \overline{F}(0 \leq n+2) \ = \ \overline{f}.$$ By Definitions \ref{bbbb} and \ref{cccc} we have for $n\geq -1$ that:
$$D_{n}\overline{f}\ = \ \mathrm{D}_{n+2}\overline{f} \ \ \ \ \ \mbox{and}\ \ \ \ \ D_{n}\overline{f}^{\natural}\ = \ \mathrm{PD}_{n+2}\overline{f} .$$
Taking isomorphisms classes of objects we get the simplicial complex $\overline{D}_{\ast}\overline{f}$ such
that $$\overline{D}_{n}\overline{f}\ = \ \overline{\mathrm{D}}_{n+2}\overline{f} \ \ \ \ \ \mbox{and} \ \ \ \ \ \overline{D}_{n}\overline{f}^{\natural} \ = \ \overline{\mathrm{PD}}_{n+2}\overline{f}.$$

\begin{thm}{\em The  M$\ddot{\mbox{o}}$bius function $\mu \in [\overline{C}_1,R]$ of a finite decomposition isomorphism filling category $C$ is given for $f$ a non-isomorphism
by $$\mu \overline{f} \ = \ \widetilde{\chi}\overline{D}_{\ast}\overline{f}. $$
}
\end{thm}

\begin{proof}From Theorem \ref{yc} we have that
$$\mu \overline{f} \ = \ \sum_{n \geq 1 }(-1)^n \big| \overline{\mathrm{PD}}_n\overline{f} \big| \ = \
\sum_{n \geq -1 }(-1)^n \big| \overline{\mathrm{PD}}_{n+2}\overline{f} \big| \ = \
\sum_{n \geq -1 }(-1)^n \big| \overline{D}_{n}\overline{f}^{\natural} \big|\ = \ \widetilde{\chi}\overline{D}_{\ast}\overline{f}.$$
\end{proof}

Note that for $\overline{x} < \overline{y} \in \overline{C}$ we have the following identity of simplicial groupoids
$$ C_{\ast}(\overline{x},\overline{y}) \ = \ \bigsqcup_{\overline{f}\in \overline{C(\overline{x},\overline{y})}}D_{\ast}\overline{f},
 \ \ \ \ \ \ \ \mbox{so we have that} \ \ \ $$
$$\widetilde{\chi}_{\mathrm{g}} C_{\ast}(\overline{x},\overline{y})\ = \ \sum_{\overline{f}\in \overline{C(\overline{x},\overline{y})}}\widetilde{\chi}_{\mathrm{g}}D_{\ast}\overline{f}. \ \ \ \
\ \ \ \ \ \ \ \ \ \ \ \ \ \ \ \ \ \ \ \ $$
From Theorem \ref{t77} we know that the M$\ddot{\mbox{o}}$bius function $\mu \in [\overline{C}_1,R]$ of $C$ is given by
$$\mu[\overline{x},\overline{y}] \ = \ \widetilde{\chi}_{\mathrm{g}} C_{\ast}(\overline{x},\overline{y}).$$
Thus we have shown the following result.

\begin{thm}\label{p}{\em Let $C$ be an essentially finite decomposition isomorphism filling category. Then $C$ is an essentially locally finite category
with  M$\ddot{\mbox{o}}$bius function $\mu(\overline{x},\overline{y})$ given for
$\overline{x} < \overline{y}\ $ in $\ \overline{C}\ $ by
$$\mu[\overline{x},\overline{y}] \  = \ \sum_{\overline{f}\in \overline{C(\overline{x},\overline{y})}}\widetilde{\chi}_{\mathrm{g}}D_{\ast}\overline{f},
 \ \ \ \ \ \ \ \mbox{where}$$
$$\widetilde{\chi}_{\mathrm{g}}D_{\ast}\overline{f} \ = \ \sum_{n \geq -1 }(-1)^n \big| D_{n}\overline{f}^{\natural} \big|_{\mathrm{g}} \ = \
\sum_{n \geq 1 }(-1)^n \big| \mathrm{PD}_n\overline{f} \big|_{\mathrm{g}}.
$$
}
\end{thm}

We close this work with an examples illustrating the meaning of Theorem \ref{p}.

\begin{exmp}{\em Consider the category  $\mathbb{I}$ of finite sets and injective maps. Clearly, we have the identity of posets  $\overline{\mathbb{I}} = \mathbb{N}$. The incidence map
$\xi \in [\overline{\mathbb{I}}_1,R]$ is given by
$$\xi[n,m]\ = \ |\mathbb{I}(\{1,...,n \}, \{1,...,m \})| \ = \ m(m-1)...(m-n+1) \ = \ \frac{m!}{(m-n)!},$$
 and the M$\ddot{\mbox{o}}$bius function $\mu \in [\overline{\mathbb{I}}_1,R]$  is given by $$\mu[n,m] \ = \  \frac{(-1)^{m-n}}{n!(m-n)!},$$ since
$$\sum_{n \leq k \leq m}\mu[n,k]\xi[k,m]  \ = \ \sum_{n \leq k \leq m}\frac{(-1)^{k-n}}{n!(k-n)!}\frac{m!}{(m-k)!} \ = \ $$
$$ {m \choose n}\sum_{s=0}^{m-n}(-1)^{s}{m-n \choose s}
\ = \  {m \choose n}(0)^{m-n}   \ = \ \delta_{n,m}.$$
The set of equivalence classes of injective maps, i.e equivalence classes of morphisms in $\mathbb{I}$, can be identified with the set
$$\{i_{n,m} \ | \ n \leq m \} ,$$
where $i_{n,m}$ is the inclusion map $$[1,n] = \{1,...,n \} \subseteq \{1,...,n,...,m \} = [1,m].$$ According to Theorem \ref{p} we must have that

$$\mu[n,m] \  = \ \widetilde{\chi}_{\mathrm{g}}D_{\ast}\overline{i}_{n,m} \ = \
\sum_{k \geq 1 }(-1)^k \big| \mathrm{PD}_k\overline{i}_{n,m} \big|_{\mathrm{g}}.$$
The proper $k$-decompositions of the morphisms $i_{n,m}$ are, up to equivalence, given by the inclusions
$$[1,n] \subseteq [1,n_1] \subseteq ...... \subseteq [1, n_{k-1}] \subseteq [1,m],$$ with
$n < n_1 < .... < n_{k-1} < m$. The automorphism group of such a decomposition has cardinality
$$\frac{1}{n!(n_1-n)!......(m-n_{k-1})!}. $$
Thus the identity $\mu[n,m] \  = \ \widetilde{\chi}_{\mathrm{g}}D_{\ast}\overline{i}_{n,m} \ $ is equivalent to
$$\frac{(-1)^{m-n}}{n!(m-n)!} \ = \
\sum_{k \geq 1 }\sum_{n < n_1...< n_{k-1} < m }\frac{(-1)^k}{n!(n_1-n)!(n_2 - n_1)!.....(m-n_{k-1})!},$$ and also to
$$(-1)^{m-n}+1 \ = \
\sum_{k \geq 2 }\sum_{n < n_1...< n_{k-1} < m }(-1)^k {m-n \choose n,\ n_1-n, \ n_2 -n_1,.....,\ m-n_{k-1}}.$$
The latter identity follows directly from  Exercise \ref{yac} below.
}
\end{exmp}

\

\begin{exc}\label{yac}{\em Fix $a \geq 2$ in $\mathbb{N}$. Show by induction that
$$\sum_{k \geq 2 }\sum_{a_1 + .... + a_k =a}(-1)^k {a \choose a_1,.....,a_k}\ = \ (-1)^a + 1,$$ where
$a_i \in \mathbb{N}_+.$ }
\end{exc}

\

\

\noindent ragadiaz@gmail.com \\
\noindent Instituto de Matem\'aticas y sus Aplicaciones\\
\noindent Universidad Sergio Arboleda, Bogot\'a, Colombia\\

\end{document}